\documentclass[a4paper]{amsart}

\pdfoutput=1

\usepackage[utf8]{inputenc}
\usepackage{lmodern}
\usepackage{amsthm, amssymb, amsmath, amsfonts, mathrsfs,dsfont,stmaryrd,url,graphicx}
\usepackage[colorlinks=true, pdfstartview=FitV, linkcolor=blue, citecolor=blue, urlcolor=blue,pagebackref=false]{hyperref}

\usepackage{microtype}

\newtheorem{theorem}{Theorem}[section]
\newtheorem{proposition}[theorem]{Proposition}
\newtheorem{lemma}[theorem]{Lemma}
\newtheorem{cor}[theorem]{Corollary}
\theoremstyle{remark}
\newtheorem{remark}[theorem]{Remark}

\let\oldH=\H

\renewcommand{\le}{\leqslant}
\renewcommand{\ge}{\geqslant}
\renewcommand{\leq}{\leqslant}
\renewcommand{\geq}{\geqslant}
\renewcommand{\subset}{\subseteq}
\renewcommand{\emptyset}{\varnothing}
\newcommand{\mcl}{\mathcal}
\newcommand{\msf}{\mathsf}

\newcommand{\msc}{\mathscr}

\newcommand{\Ll}{\left}
\newcommand{\Rr}{\right}

\newcommand{\rhs}{right-hand side}
\newcommand{\1}{\mathds{1}}
\newcommand{\N}{\mathbb{N}}
\newcommand{\R}{\mathbb{R}}

\newcommand{\Z}{\mathbb{Z}}
\newcommand{\ov}{\overline}

\newcommand{\td}{\widetilde}
\newcommand{\eps}{\varepsilon}
\renewcommand{\d}{{\mathsf{d}}}

\newcommand{\GN}{{\mathcal{G}_N}}
\newcommand{\EN}{{\mathcal{E}_N}}
\renewcommand{\P}{\mathbb{P}}
\newcommand{\Pb}{\mathbb{P}^{b,p}_{\gamma}}
\newcommand{\E}{\mathbb{E}}

\newcommand{\al}{\alpha}

\newcommand{\ga}{\gamma}

\newcommand{\si}{\sigma}
\renewcommand{\H}{\mathcal{H}}
\newcommand{\C}{\mathcal{C}}

\numberwithin{equation}{section}

\title{Spatial Gibbs random graphs}
\author{Jean-Christophe Mourrat, Daniel Valesin}

\address[Jean-Christophe Mourrat]{Ecole normale sup\'erieure de Lyon, CNRS, Lyon, France}

\address[Daniel Valesin]{University of Groningen, Netherlands}

\begin{document}

\begin{abstract}

Many real-world networks of interest are embedded in physical space. We present a new random graph model aiming to reflect the interplay between the geometries of the graph and of the underlying space. The model favors configurations with small average graph distance between vertices, but adding an edge comes at a cost measured according to the geometry of the ambient physical space. In most cases, we identify the order of magnitude of the average graph distance as a function of the parameters of the model. As the proofs reveal, hierarchical structures naturally emerge from our simple modeling assumptions. Moreover, a critical regime exhibits an infinite number of discontinuous phase transitions.

\medskip

\noindent \textsc{MSC 2010: 82C22; 05C80} 

\medskip

\noindent \textsc{Keywords:} spatial random graph, Gibbs measure, phase transition.

\end{abstract}
\maketitle
%
%
%
%
%
%
%
%
\section{Introduction}

In the Erd\oldH{o}s-Rényi random graph, pairs of nodes are connected independently and with the same probability. It is now well-known that most networks of interest in biological, social and technological contexts depart a lot from this fundamental model. In a very influential paper \cite{barabasi1999}, Barab{\'a}si and Albert suggested that these more complex networks have in common that their degree distributions seem to follow a power law. This is in stark contrast with the degree distribution observed in Erd\oldH{o}s-Rényi graphs, which has finite exponential moments. They proposed that this property become the signature of complex networks, a sort of ``order parameter'' of these systems. They then observed that a growth mechanism with preferential attachment reproduces the power-law behavior of the degree distribution.
The work of Barab{\'a}si and Albert triggered a lot of activity, in particular on preferential attachment rules and the configuration model. We refer to \cite{vdH} for a comprehensive account of the mathematical activity on the subject. 

\smallskip

This point of view is however not all-encompassing \cite{keller2005}. Several studies point to the fact that different graphs may share the same degree distribution, and yet have very different large-scale geometries; and moreover, that the ``entropy maximizing'' graphs with a power-law degree sequence---those that would be favored by the point of view expressed above---actually do \emph{not} resemble certain real-world networks. For instance, the authors of \cite{li2004} show that the physical infrastructure of the Internet is very far from resembling a graph obtained from the dynamics of preferential attachment; instead, hierarchical structures are observed, and the organization of the network is best explained as the result of some optimization for performance (see in particular \cite[Figures~6 and 8]{li2004}). Similarly, the network of synaptic connections of the brain depart a lot from ``maximally random'' graphs with a power-law degree sequence \cite{sporns2001}. They also exhibit a hierarchical organization, as well as high clustering, and the authors of \cite{sporns2001} suggest that this is the result of an attempt to maximize a certain measure of complexity of the network, with a view towards computational capabilities (see also \cite{sporns2000a,sporns2000b,sporns2004}). 

\smallskip

The goal of the present paper is to introduce a new model of random graph which is hopefully more representative of such real-world graphs. In our view, one fundamental requirement for our model is to retain the fact that graphs such as the infrastructure of the Internet, transportation or neural networks, are embedded in physical space. The examples we described above seem to suggest that the graphs of interest are the result of some optimization: for the efficient transportation of information in the case of the infrastructure of the Internet, or for some notion of complexity for neural networks. In fact, it is very easy to imagine a wealth of other natural objective functions for a network, depending on the context. As for the geometry of the underlying space, it would be natural to take it as a large subgraph of $\Z^d$. Here we restrict our attention to a one-dimensional underlying structure. 
As for the objective function, we chose a measure of connectedness of the graph: minimizing the diameter of the graph is an example of objective we consider. 

\smallskip

One of the key findings of our study is that despite its simplicity, our model displays a very rich variety of behavior. In particular, a critical case displays an infinite number of discontinuous phase transitions. Moreover, hierarchical structures emerge spontaneously, in the sense that they are not built into the definition of the model. As was pointed out above, hierarchical structures have been seen to occur in real-world networks. While these hierarchies were assumed to emerge from technological constraints in \cite{li2004} (in particular, because only a handful of routers with different bandwidths are commercially available), we show here that the requirements of optimization of the objective function can be sufficient to account for the emergence of such structures.


\smallskip

The random graph we study is the result of a balance between a desire to optimize a certain objective function and entropy effects. 
As announced, we wish to focus here on the simplest possible such model, and therefore restrict ourselves to a one-dimensional ambient space. Let $N$ be a positive integer, and let $G^\circ_N = (V_N,E^\circ_N)$ be the graph with vertex set 
$$
V_N = \{0,\ldots, N-1\}
$$ 
and edge set 
$$
E^\circ_N = \{ \{x,x+1\} \ : \  x,x+1 \in V_N\}.
$$ 
We will refer to elements of $E^\circ_N$ as \emph{ground edges}. In analogy with a transportation network, we may think of elements of $V_N$ as towns, and of edges in $E^\circ_N$ as a basis of low-speed roads connecting towns in succession. We now consider the possibility of adding additional edges ``above'' the ground edges, which we may think of as faster roads, or flight routes. Let 
$$
\msc{E}_N = \Ll\{ \{x,y\} \ : \ x \neq y \in V_N \Rr\} 
$$
be the set of (unordered) pairs of elements of $V_N$, and
$$
\msc G_N = \{g = (V_N,  E_N) : E^\circ_N \subset E_N \subset \msc E_N\}
$$
be the set of graphs over $V_N$ that contain $G^\circ_N$ as a subgraph. 
Each graph $g = (V_N, E_N) \in \msc G_N$ induces a graph metric given by
\begin{multline*}
\d_{g}(x,y) 
= \inf\big\{k \in \N : \exists x_0=x, x_1, \ldots, x_{k-1}, x_k = y \text{ s.t.} \\
\text{ for all } 0 \le j < k, \{x_j,x_{j+1}\} \in  E_N\big\}.
\end{multline*}
This distance is not to be confused with the ``Euclidean'' distance $|\cdot|$. 
For a given $p \in [1,\infty]$ and for each $g \in \mcl G_N$, we define the $\ell^p$-\emph{average path length} by
\begin{equation}
\label{e.def.H}
\H_p(g) = \Ll(\frac{1}{N^2} \sum_{x,y \in V_N} \d_g^p(x,y)\Rr)^\frac 1 p,
\end{equation}
with the usual interpretation as a supremum if $p = \infty$. (In other words, $\H_\infty(g)$ is the diameter of the graph $g$.)
We would like to minimize this average path length, subject to a ``cost'' constraint. The cost is defined in terms of a parameter $\ga \in (0,\infty)$ by
\begin{equation*}  
\C_\ga(g) = \sum_{e \in E_N \atop |e| > 1} |e|^\ga,
\end{equation*}
where for each edge $e = \{x,y\} \in \msc E_N$, we write $|e| = |y-x|$ for the length of the edge $e$. When $\ga = 1$, the cost of a link is equal to its length; the case $\ga < 1$ can be thought of as a situation with ``economies of scale'', in which the marginal cost of an edge is lower when the edge is longer. 

\smallskip

Ideally, we would wish to find the graph $g$ minimizing $\mcl H_p(g)$ subject to a given upper bound on the cost function $\C_\ga(g)$. However, real-life constraints prevent this optimization problem from being resolved exactly. Instead, the resulting graph will be partly unpredictable, and we assume that its probability distribution  follows the Gibbs principle. In other words, we are interested in the Gibbs measure with energy given by a suitable linear combination of $\mcl H_p(g)$ and $\C_\ga(g)$. 

\smallskip

In order to simplify a little the ensuing analysis, we define our model in a slightly different way. We denote the canonical random graph on $\msc G_N$ by $\GN = (V_N, \EN)$. For each $\ga \in (0,\infty)$, we give ourselves a reference measure $\P_\ga$ on $\msc G_N$ such that under $\P_\ga$,
\begin{multline}
\label{e.indep}
\mbox{the events $(\{e \in \EN\})_{e \in \msc E_N, |e| > 1}$ are independent}, \\
\mbox{and each event has probability } \exp\Ll(-|e|^\ga \Rr).
\end{multline}
(We do not display the dependency on $N$ on the measures $\P_\ga$; we may think of the latter as a measure on $\prod_N \msc G_N$.)  We denote by $\E_\ga$ the associated expectation.
Then, for each given $b \in \R$ and $p \in [1,\infty]$, we consider the probability measure $\P^{b,p}_{\ga}$ such that for every $g \in\msc G_N$,
\begin{equation}
\label{e.def.P}
\P^{b,p}_{\ga}[\GN = g] = \frac 1 {Z^{b,p}_{\ga,N}} \exp \Ll( -N^b \, \H_p(g)\Rr) \, \P_{\ga}[\GN = g],
\end{equation} 
where the constant $Z^{b,p}_{\ga,N}$ ensures that $\P^{b,p}_{\ga}$ is a probability measure:
\begin{equation}
\label{e.def.Zbp}
Z^{b,p}_{\ga,N} = \E_{\ga} \Ll[\exp \Ll( -N^b \, \H_p(\GN)\Rr)\Rr].
\end{equation}
We denote by $\E^{b,p}_{\ga}$ the expectation associated with $\P^{b,p}_{\ga}$. One can check that the measure $\P^{b,p}_{\ga}$ is the Gibbs measure with energy
\begin{equation*}  
N^b \H_p(g) - \sum_{e \in E_N} \log \Ll( \frac{\exp \Ll( -|e|^\ga \Rr) }{1-\exp\Ll(-|e|^\ga\Rr)} \Rr) ,
\end{equation*}
which is a minor variant of the energy $N^b \H_p(g) + \C_\ga(g)$. A natural extension of our model would be to consider energies of the form
$$
\beta_N \H_p(g) + \lambda_N \C_\gamma(g),
$$
for general sequences $(\beta_N)$ and $(\lambda_N)$. However, this increase in generality does not seem to change the qualitative behavior of the model, so we favored clarity over generality. 

\smallskip

Our first main result characterizes the behavior of the average path length in terms of the parameters $\gamma$, $b$ and $p$ when $\gamma \neq 1$. 

\begin{theorem}
\label{t.main1}
For every $\ga \neq 1$ and $b \in \R$, let
\begin{equation*}  
\al(\ga,b) := \Ll|
\begin{array}{ll}
\displaystyle{\Ll(\frac{1-b}{2-\ga} \wedge 1\Rr) \vee 0} & \mbox{ if $\ga < 1$},\\[12pt]
\displaystyle{\Ll(\frac{\ga - b}{\ga} \wedge 1\Rr) \vee 0} & \mbox{ if $\ga > 1$}.
 \end{array}
 \Rr.
\end{equation*}
For every $\ga \neq 1$, $b \in \R$, $p \in [1,\infty]$ and $\eps > 0$, we have
\begin{equation*}  
\lim_{N \to \infty}\;  \P^{b,p}_{\ga}\left[\left|\frac{\log \H_p(\mathcal{G}_N)}{\log N} - \al(\ga,b)\right| > \varepsilon \right] = 0.
\end{equation*}
\end{theorem}
Drawings of the function $b \mapsto \al(\ga,b)$ in the cases $0 < \gamma < 1$ and $\gamma > 1$ are displayed in Figure~\ref{f.alpha}.

\begin{figure}[htb]
\begin{center}
\setlength\fboxsep{0pt}
\setlength\fboxrule{0.5pt}
\fbox{\includegraphics[width = 1.0\textwidth]{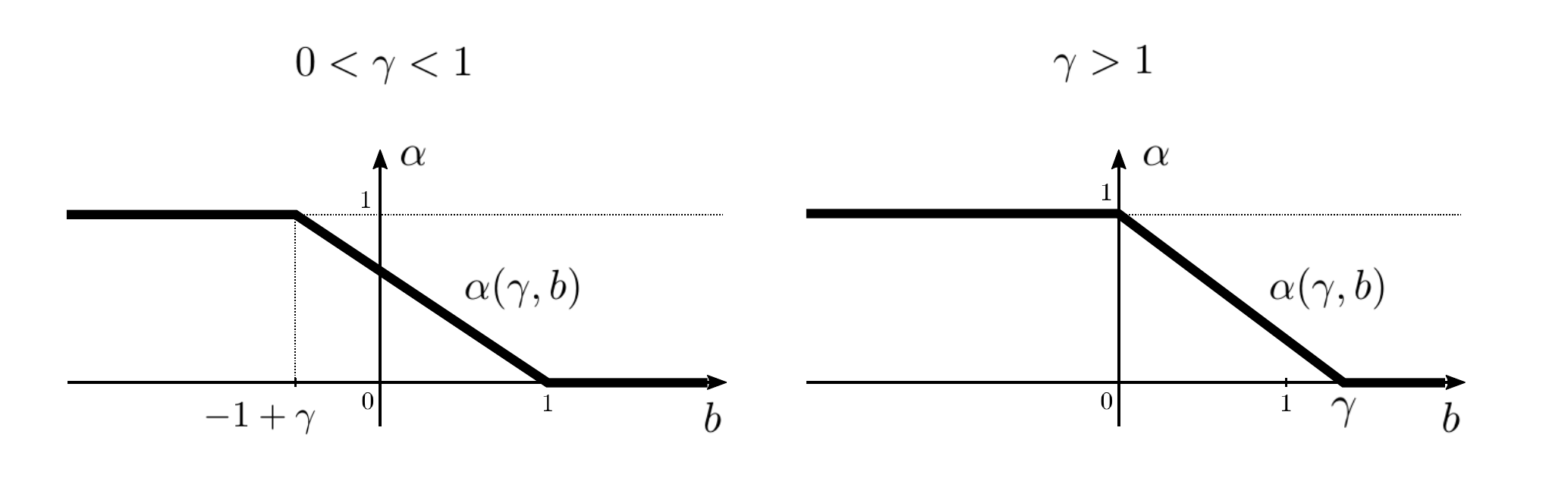}}
\end{center}
\caption{{\small Under $\P^{b,p}_\ga$ for $\ga \neq 1$, we have $\frac{\log \H_p(\mathcal{G}_N)}{\log N} \approx \al(\ga,b)$ with high probability. }}
\label{f.alpha}
\end{figure}

\smallskip

The proof of Theorem~\ref{t.main1} essentially reduces to showing that under the reference measure $\P_\ga$, for every $p \in [1,\infty]$ and $\al \in (0,1)$, one has
\begin{equation}
\label{e.reference}
-\log \P_\ga \Ll[ \H_p(\GN) \simeq N^{\al} \Rr] \simeq 
\Ll|
\begin{array}{ll}
N^{1-\al(1-\gamma)}  & \mbox{ if  $\ga < 1$}, \\
N^{1 + (1-\al)(\ga - 1)}  & \mbox{ if  $\ga > 1$}.
\end{array}
\Rr.
\end{equation}
For $\ga < 1$, the lower bound for this probability is obtained by the hierarchical construction depicted in the top graph of Figure~\ref{f.lower}: we draw the edge connecting the extremities of the interval $V_N$, then the two edges connecting each extremity with the middle point of $V_N$, and so on recursively until reaching edges of length $N^\al$. The lower bound for the case $\ga > 1$ is obtained similarly, but starting from edges of length~$2$ and building successive layers of larger edges, as depicted in the bottom graph in Figure~\ref{f.lower}, until we reach edges of size $N^{1-\al}$.

\begin{figure}[htb]
\begin{center}
\setlength\fboxsep{0pt}
\setlength\fboxrule{0.5pt}
\fbox{\includegraphics[width = 1.0\textwidth]{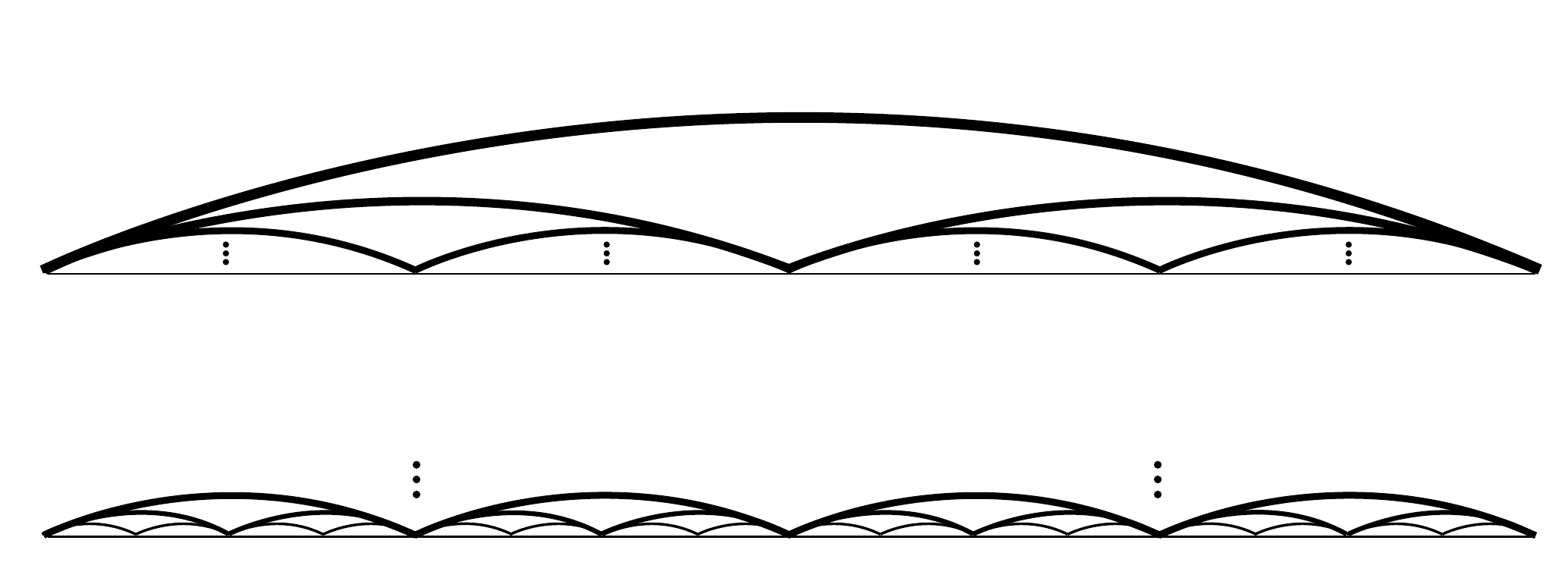}}
\end{center}
\caption{{\small Hierarchical constructions that provide lower bounds for Theorem \ref{t.main1}: case $\gamma < 1$ (top) and $\gamma > 1$ (bottom). }}
\label{f.lower}
\end{figure}

\smallskip

The proof of the upper bound for the left-hand side of \eqref{e.reference} confirms the relevance of the strategy used in the proof of the lower bound in the following sense. For $\ga > 1$, we show that outside of an event of probability smaller than the right side of \eqref{e.reference}, there are of order $N^{\al}$ points at Euclidean distance at least $N^{1-\al}$ from one another and such that no edge of length $N^{1-\al}$ or more goes ``above'' any of these points. For $\ga < 1$, outside of an event of suitably small probability, we identify about $N^{1-\al}$ disjoint sub-intervals which are each of diameter $N^{\al}$ and have no direct connection between one another.

\smallskip

Our second main result concerns the case $\ga = 1$. This case is critical, and therefore more difficult.
Rather than ``all $b \in \mathbb{R}$ and all $p \in [1,\infty]$'' (as in the statement of Theorem \ref{t.main1}), Theorem \ref{t.main2} is applicable to a certain set of  $(b,p) \in \mathbb{R} \times [1,\infty]$. This set is shown in Figure \ref{fig:critical} and defined by
\begin{equation}\label{eq:condition_on_bp}
\frac{k-1}{k} + h(k,p) < b < \frac{k}{k+1}\qquad  \text{for some } k \in \N,
\end{equation}  
where $h:\mathbb{N}\times[1,\infty] \to \mathbb{R}$ is defined by
\begin{equation}\label{eq:def_of_h}
h(k,p):= \Ll|\begin{array}{cl}  \frac{2p-(p-1)k}{k(k+1)(k+2p)}\vee 0&\mbox{if } p<\infty;\\[.2cm]
\frac14&\mbox{if } p=\infty \text{ and } k = 1;\\[.2cm]
0&\mbox{if } p = \infty \text{ and } k > 1. \end{array} \Rr.
\end{equation}

\begin{figure}[htb]
\begin{center}
\setlength\fboxsep{0pt}
\setlength\fboxrule{0.5pt}
\fbox{\includegraphics[width = 1.0\textwidth]{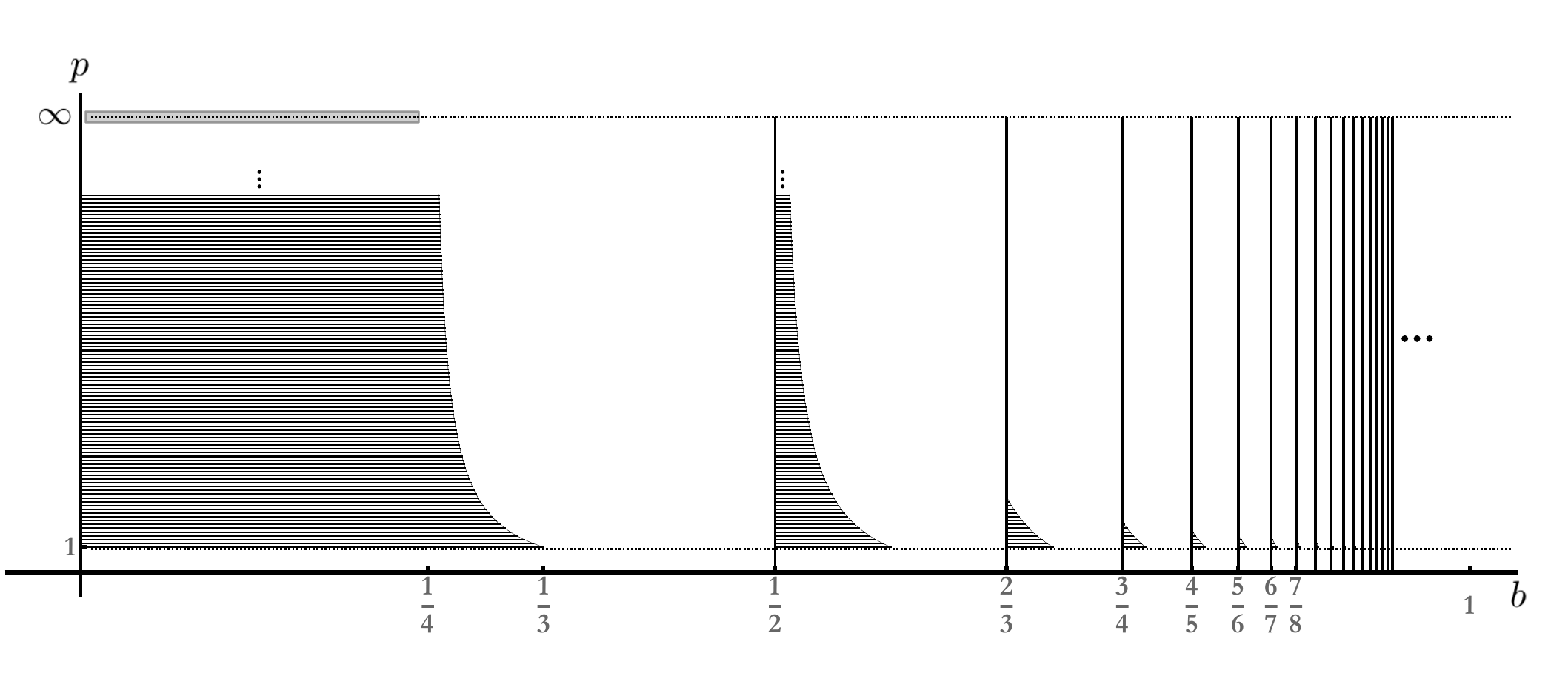}}
\end{center}
\caption{{\small Each rectangular region delimited by the horizontal lines $p = 1$ and $p=\infty$ and the vertical lines $b = \frac{k-1}{k}$ and $b = \frac{k}{k+1}$ is divided into a dark part and a white part. The white part consists of the values of $(b,p)$ covered by Theorem \ref{t.main2}. 
}}
\label{fig:critical}
\end{figure}
\begin{theorem}
\label{t.main2}
If $p \in [1,\infty]$, $k \in \N$,  $b \in \R$ satisfy
$
\frac{k-1}{k} + h(k,p) < b < \frac{k}{k+1},
$
and $\eps > 0$, we have
\begin{equation*}
\lim_{N \to \infty}\;  \P^{b,p}_{1}\left[\left|\frac{\log \H_p(\mathcal{G}_N)}{\log N} - \frac{1}{k+1}\right| > \varepsilon \right] = 0.
\end{equation*}
\end{theorem}
 We note in particular that for each $p > 1$, we have $h(k,p) > 0$ if and only if $k < \frac{2p}{p-1}$. Therefore, for each $p > 1$, Theorem~\ref{t.main2} guarantees an infinite number of discontinuous transitions for
\begin{equation*}  
\lim_{N \to \infty} \frac{\log \H_p(\mathcal{G}_N)}{\log N},
\end{equation*}
which ultimately spans the sequence $\Ll(\frac 1 k\Rr)_{k \in \N}$ as $b$ increases to $1$. Figure \ref{f.alpha_c} displays this phenomenon more precisely, and is in sharp contrast with the naive continuation of the graphs of Figure~\ref{f.alpha} to the value $\ga = 1$.

\begin{figure}[htb]
\begin{center}
\setlength\fboxsep{0pt}
\setlength\fboxrule{0.5pt}
\fbox{\includegraphics[width = \textwidth]{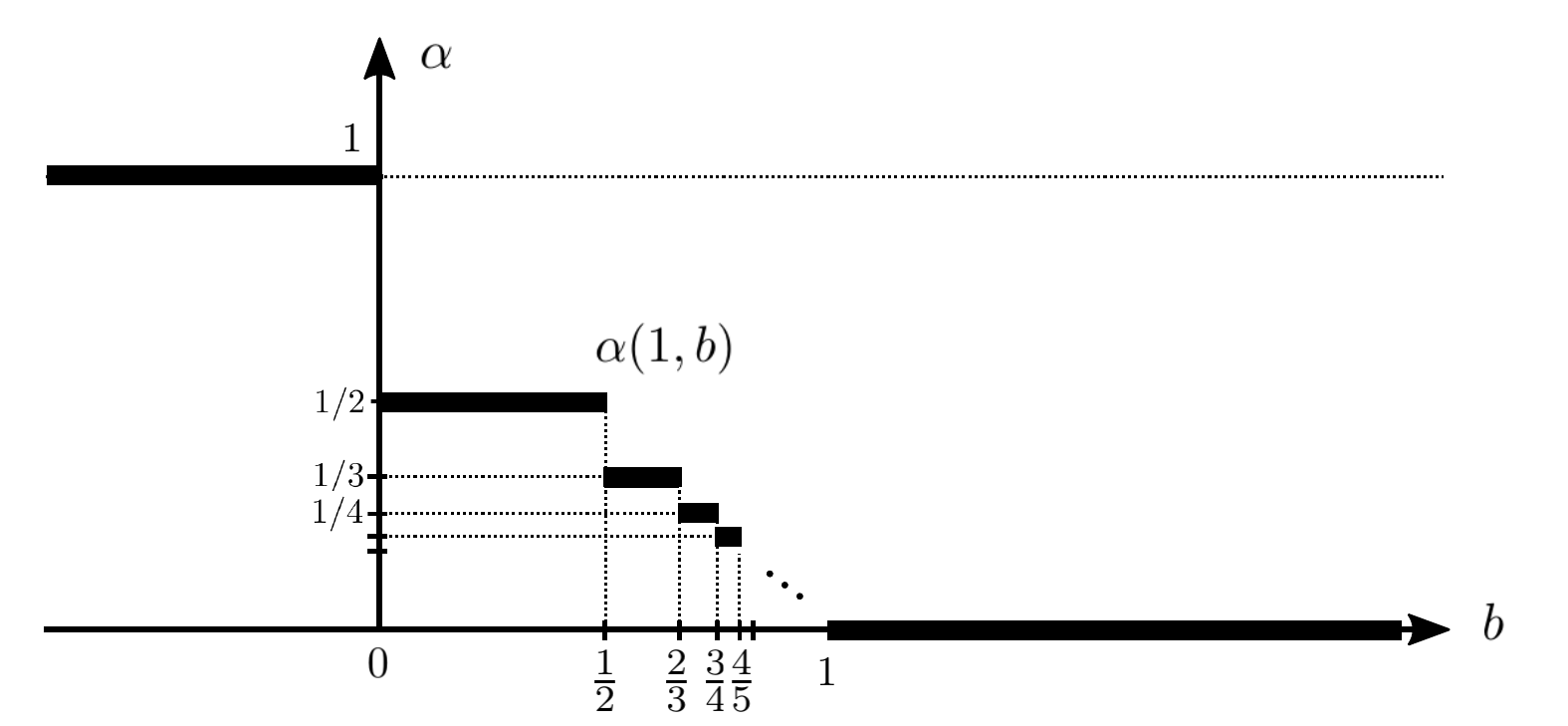}}
\end{center}
\caption{{\small The function $\alpha(1,b)$ plotted above is such that, if $b$ and $p$ satisfy the condition given in \eqref{eq:condition_on_bp}, then under $\P^{b,p}_1$ we have $\frac{\log \H_p(\mathcal{G}_N)}{\log N} \approx \al(1,b)$ with high probability as $N \to \infty$. }}
\label{f.alpha_c}
\end{figure}


\smallskip

The origin of this phenomenon can be intuitively understood as follows. Irrespectively of the value of $\gamma$, the only efficient strategies for reducing the average path length consist in the addition of successive layers of edges above $E_N^\circ$, each of which essentially covers the interval $\{1,\ldots,N\}$. When $\ga = 1$, all layers covering $\{1,\ldots,N\}$ without redundancy have the same cost. If only one layer is allowed, then the most distance-reducing layer is one made of edges of length $N^{\frac 1 2}$, which brings the average path length down to about $N^\frac 1 2$. If two layers are allowed, then it is best to choose one made of edges of length $N^{\frac 1 3}$, and one made of edges of length $N^\frac 2 3$, in which case the average path length is about $N^\frac 1 3$. If $k$ coverings are allowed, then we use layers made of edges of length $N^{\frac 1{k+1}},N^{\frac 2{k+1}}\ldots, N^{\frac k{k+1}}$ respectively, so as to reduce the average path length to $N^\frac 1 k$. The graphs for the cases $k = 1$ and $k = 2$ are illustrated in Figure~\ref{f.crit_graph}. Note that these graphs may be seen as ``in between'' those displayed at the top and bottom of Figure~\ref{f.lower}.

\smallskip

\begin{figure}[htb]
\begin{center}
\setlength\fboxsep{0pt}
\setlength\fboxrule{0.5pt}
\fbox{\includegraphics[width = 1.0\textwidth]{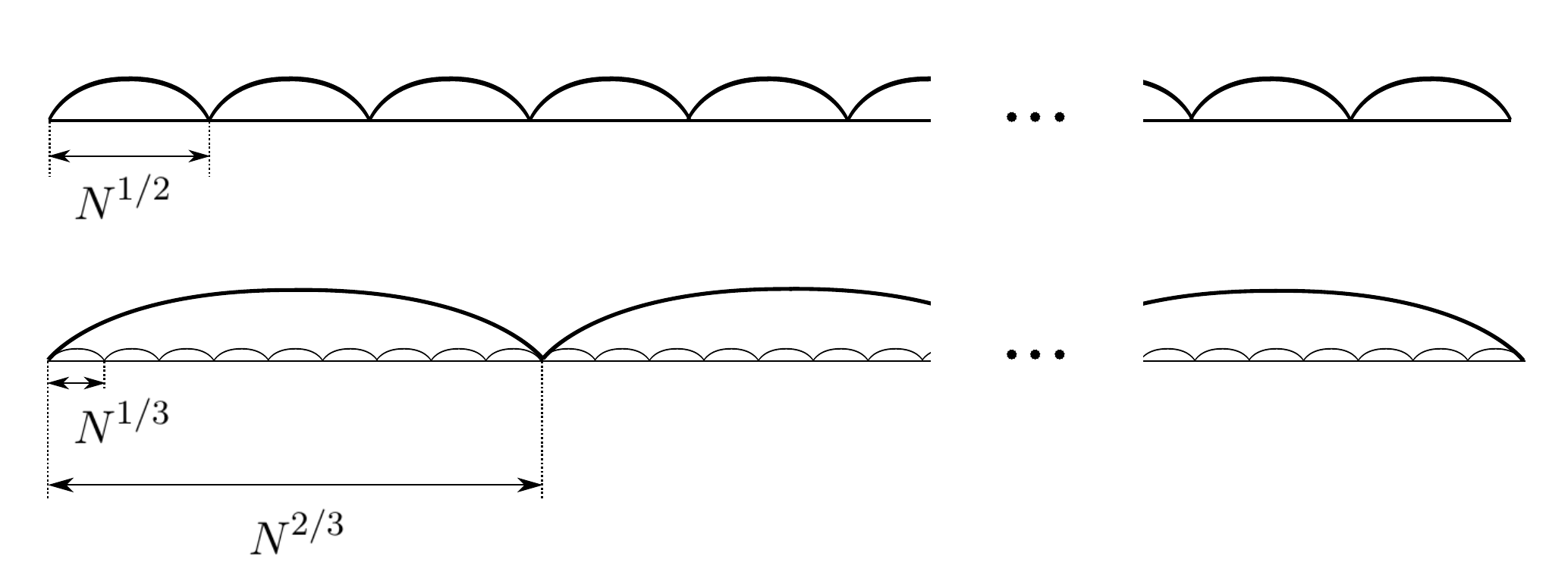}}
\end{center}
\caption{{\small Hierarchical constructions that provide lower bounds for Theorem~\ref{t.main2}. }}
\label{f.crit_graph}
\end{figure}

In view of this, the proof of Theorem~\ref{t.main2} will necessarily be more involved than that of Theorem~\ref{t.main1}. Indeed, 
in the limiting case $\ga = 1$, the right side of \eqref{e.reference} no longer depends on $\al$. The estimate is therefore no longer discriminative, and the proof of Theorem~\ref{t.main2} must rely on more precise information on the probability of deviations of $\H_p(\GN)$ under the reference measure~$\P_1$. Our argument is faithful to the intuition described above, in that we inductively ``reveal'' the necessity of the existence of these successive layers.

\begin{remark}  
We conjecture that Theorem~\ref{t.main2} holds with $h \equiv 0$. Although we do not prove this, our proof of the theorem provides some extra information concerning values of $(b,p)$ that do not satisfy \eqref{eq:condition_on_bp}. Namely,
\begin{enumerate}
\item for every $p \in [1,\infty]$, $b < 0$ and $\eps > 0$, we have
\begin{equation*}\lim_{N \to \infty}\;  \P^{b,p}_{1}\left[\frac{\log \H_p(\mathcal{G}_N)}{\log N} < 1-\eps \right] = 0;\end{equation*}
\item for every $p \in [1,\infty]$, $k \in \N$, $b \in (\frac{k-1}{k},\frac{k}{k+1})$ and $\eps > 0$,
\begin{equation*}\lim_{N \to \infty}\;  \P^{b,p}_{1}\left[\frac{1}{k+1}-\eps< \frac{ \log \H_p(\mathcal{G}_N)}{\log N} < \frac1k -\eps \right] = 1;\end{equation*}
\item for every $p \in [1,\infty]$, $b > 1$ and $\eps > 0$, we have
\begin{equation*}\lim_{N \to \infty}\;  \P^{b,p}_{1}\left[\frac{\log \H_p(\mathcal{G}_N)}{\log N} > \eps \right] = 0.\end{equation*}

\end{enumerate}
\end{remark}

Theorems~\ref{t.main1} and \ref{t.main2} demonstrate that random graph models that are embedded in some ambient space, and that relate to the minimization of some objective function, are amenable to mathematical analysis. They offer a glimpse of some features of real-world networks not captured by more common models, in particular with naturally emerging hierarchical structures. Of course, these results also call for improvement: besides closing the gap apparent in Theorem~\ref{t.main2}, it would be very interesting to obtain more specific results about the exact structure of the hierarchies we expect to be present in the graph. We point out that it is not straightforward to see them appearing in simulations of Glauber-type dynamics adapted to the model we study. We are grateful to Vincent Vigon (University of Strasbourg) for performing such simulations, which are accessible at
\begin{center}
\small{\url{http://mathisgame.com/small_projects/SpacialGibbsRandomGraph/index.html}}
\end{center}

\smallskip

It would also be very interesting to explore generalizations of the model. For many real-world networks, it would be most natural to consider an underlying geometry given by a large box of $\Z^d$, $d \in \{2,3\}$, as opposed to the case $d = 1$ considered here. In fact, the model we consider could be defined starting from an arbitrary reference graph $G^\circ$: the cost of the addition of an edge would then be a function of the distance in the original graph $G^\circ$. Ideally, one would then aim to determine how the properties we discussed here depend on the geometry of the graph $G^\circ$.

\smallskip

Another possible direction for future work would be to consider other objective functions to minimize. We already mentioned that a certain measure of ``complexity'' was identified as a parameter to optimize for neural networks; and that the efficient transportation of information is certainly an explanatory variable for the physical structure of the Internet. Many variations can be imagined. For instance, one may assume that in order to turn a vertex into an efficient ``hub'' with many connections to other vertices, one needs to pay a certain cost (e.g.\ because more infrastructure is necessary, a more powerful router needs to be bought and installed, etc.). This assumption may strengthen the possibility of degree distributions having a fat polynomial tail.

\smallskip




\smallskip

One of the implicit assumptions in our model is that the vertices in $V_N$ are all given the same importance in the computation of the average path length. If we think of the vertices of $V_N$ as towns, it would be more natural to weigh the average path length according to some measure of the number of inhabitants in each town. That is, we would endow each $x \in V_N$ with a number $\tau_x$ measuring the ``importance'' of the vertex $x$, and replace $\H_p(g)$ by a suitable multiple of
\begin{equation}  
\label{e.disordered}
\Ll(\sum_{x,y \in V_N} \tau_x  \tau_y \, \d_g^p(x,y)\Rr)^\frac 1 p.
\end{equation}
As is well-known, city size distributions follow a power law, as do a wide range of other phenomena \cite{yule1925,simon1,simon2}. In this disordered version of our model, it would therefore be natural to assume that $(\tau_x)$ are i.i.d.\ random variables with a power-law tail.

\smallskip

We conclude this introduction by mentioning related works. First, as was apparent in \eqref{e.reference}, our results can be entirely recast in terms of large deviation estimates for some long-range percolation model. While this point of view is also natural, we prefer to emphasize the point of view based on Gibbs measures, which motivates the whole study (and explains in particular our need for a very fine control of the next-order correction to \eqref{e.reference} in the critical case $\gamma =1$, see Proposition~\ref{p.crit} below). For long-range percolation models, it is natural to assume a power-law decay of the probability of a long connection. In contrast, under the reference measure $\P_\gamma$ of our model, we recall that the probability of presence of an edge of length $|e|$ decays like $\exp \Ll( -|e|^\gamma \Rr)$ instead; power-law behavior of long connections is only expected under the Gibbs measure, and for the right choice of parameters. Early studies in long-range percolation models include \cite{schulman1983, newsch,aiznew,akn,gkn}, and were mostly focused on the existence and uniqueness of an infinite percolation cluster. The order of magnitude of the typical distance and the diameter for such models was studied in \cite{benber, copper, biskup, biskup2, ding}. The variant of our model discussed around \eqref{e.disordered} is reminiscent of the inhomogeneous, long-range percolation model introduced in \cite{dvh}. We are not aware of previous work on large deviation events for long-range percolation models.

\smallskip

With aims comparable to ours, several works discussed models obtained by modulating the rule of preferential attachment by a measure of proximity, see \cite{flaxman,aiello,jordan, janssen,cooper,jm1,jm2}. The survey \cite{aldshu} is a good entry point to the literature on geometric and proximity graphs, where for example one draws points at random in the plane and connects points at distance smaller than a given threshold. Upper and lower bounds in problems of balancing short connections and costs of routes were obtained in \cite{aldken, ald}. Similar considerations led to the definition of certain ``cost-benefit'' mechanisms of graph evolution in \cite{louf,mengistu,verma}. 
Another line of research is that of exponential random graphs, see for instance \cite{chadia}, where Gibbs transformations of random graphs such as the configuration model are studied. (We are not aware of spatially embedded versions of these models.) Yet another direction is explored in \cite{achsim}, where the authors give conditions ensuring that the uniform measure on a set of graphs satisfying some constraints can be well-approximated by a product measure on the edges. 

\medskip

\noindent \textbf{Organization of the paper.} We prove Theorem~\ref{t.main1} in Section~\ref{s.neq1}, and Theorem~\ref{t.main2} in Section~\ref{s.gamma1}. The appendix contains a classical large deviation estimate, which we provide for the reader's convenience.

\smallskip

\noindent \textbf{Terminology.} We call any set of the form $\{a,\ldots,b\}$ with $a,b \in V_N$, $a < b$ an \emph{integer interval}. Whenever no confusion occurs, as in this introduction, we simply call it an \emph{interval}.

%
%
%
%
%
%
%
\section{\texorpdfstring{Case $\ga \neq 1$}{Case gamma neq 1}}
\label{s.neq1}
The goal of this section is to prove Theorem~\ref{t.main1}. The section is split into three subsections: we first prove respectively lower and upper bounds on the probability of deviations of $\H_p(\GN)$ under the reference measure $\P_\ga$, and then use them to conclude the proof in the last subsection.
\subsection{Lower bounds}
\label{ss.lower}


In this subsection, we prove lower bounds on the probability of deviations of the diameter~$\H_\infty(\GN)$ under the reference measure $\P_\ga$.
\begin{proposition}
\label{p.lower}
\emph{(1)} If $\ga < 1$, then there exists $C < \infty$ such that for every $\alpha \in (0,1)$, 
$$
\P_\ga[\H_\infty(\GN) \le N^\al] \ge \exp\Ll( -C N^{1-\al(1-\ga)} \Rr).
$$
\emph{(2)} If $\ga > 1$, then there exists $C < \infty$ such that for every $\alpha \in (0,1)$,
$$
\P_\ga[\H_\infty(\GN) \le N^\al] \ge \exp\Ll( -C N^{1+(1-\al)(\ga-1)} \Rr).
$$
\end{proposition}
\begin{proof}
For $1 < k \le l$, let 
\begin{equation}
\label{e.def.Lambda}
E_N(k,l) := \big\{ \{i 2^j, (i+1) 2^j\} \in \msc E_N \ : \ i \in \N, \ k \le j \le l \big\}.
\end{equation}
We denote by $\mcl A_N(k,l)$ the event that $E_N(k,l) \subset \EN$.

Let $n$ be the largest integer such that $2^n < N$, and let $k \le n$. When $\ga < 1$, the most efficient strategy for reducing the diameter $\H_\infty$ is to start building a binary hierarchy starting from the highest levels. We are therefore interested in showing that
\begin{equation}
\label{e.low-small-imp}
\mcl A_N(k,n) \quad \implies \quad  \H_\infty(\GN) \le 2^{k+1} + 2(n-k).
\end{equation}
Let $x \in V_N$. For $i_k := \lfloor x/2^k \rfloor$, we have $i_k2^k \in V_N$ and $|x-i_k2^k| < 2^k$. We then define inductively, for every $l \in \{k+1, \ldots, n\}$,
$$
i_{l+1} = \lfloor i_l/2 \rfloor.
$$
We observe that $i_{n+1} = 0$ and
$$
|i_l 2^l - i_{l+1}2^{l+1}| \in \{0, 2^{l} \},
$$
so either the edge $\{i_l 2^l, i_{l+1} 2^{l+1}\}$ belongs to $E_N(k,n)$, or the endpoints are equal. On the event $\mcl A_N(k,n)$, the following path connects $x$ to~$0$ and belongs to $\GN$: take less than $2^k$ unit-length edges to go from $x$ to $i_k 2^k$, and then follow the edges $\{i_l 2^l, i_{l+1} 2^{l+1}\}$ (when the endpoints are different) until reaching $0$ for $l = n$. The total number of steps in this path is less than
$2^k + (n-k)$.
Hence, on the event $\mcl A_N(k,n)$, any two points can be joined by a path of length at most twice this size, and this proves~\eqref{e.low-small-imp}.

It follows from \eqref{e.low-small-imp} that
$$
\P_\ga[\H_\infty(\GN) \le 2^{k+1} + 2(n-k)] \ge \P_\ga[\mcl A_N(k,n)].
$$
In view of what we want to prove and of the fact that $n < \log_2(N)$, we fix $k$ to be the largest integer such that $2^{k} \le N^\al/4$. Since $n < \log_2(N)$, for $N$ sufficiently large, for this choice of $k$, we have
\begin{equation*}  
\P_\ga[\H_\infty(\GN) \le N^\al] \ge \P_\ga[\mcl A_N(k,n)].
\end{equation*}
By \eqref{e.indep} and the fact that $\ga < 1$, the probability in the \rhs\ is
\begin{align*}
\prod_{j = k}^n  \Ll( {\exp\Ll(-2^{\ga j} \Rr)}\Rr)^{\lfloor (N-1)/2^j \rfloor} 
& \ge \exp \Ll(-C N   2^{-(1-\ga)k} \Rr) \\
& \ge \exp \Ll(-C N^{1-\al(1-\ga)} \Rr) ,
\end{align*}
where $C < \infty$ may change from line to line, and where we used the definition of $k$ in the last step. This completes the proof of part (1) of the proposition.

\medskip

We now turn to part (2) of the proposition. When $\ga > 1$, it is more efficient to use events of the form $\mcl A_N(1,k)$ for a suitably chosen $k$. Indeed, similarly to \eqref{e.low-small-imp}, one can show
\begin{equation}
\label{e.low-large-imp}
\mcl A_N(1,k) \quad \implies \quad \mcl H_\infty(\GN) \le 2^{n-k+2} + 2k,
\end{equation}
and therefore
$$
\P_\ga[\H_\infty(\GN) \le 2^{n-k+2} + 2k] \ge \P_\ga[\mcl A_N(1,k)].
$$
We choose $k$ to be the smallest integer such that $2^{n-k} \le N^\al/8$. (Recall that by the definition of $n$, this roughly means $2^k \simeq N^{1-\al}$.) For this choice of $k$ and $N$ sufficiently large, we have
\begin{equation*}  
\P_\ga[\H_\infty(\GN) \le N^\al] \ge \P_\ga[\mcl A_N(1,k)].
\end{equation*}
The latter probability is equal to
\begin{align*}
\prod_{j = 1}^k  \Ll({\exp\Ll(-2^{\ga j} \Rr)}\Rr)^{\lfloor (N-1)/2^j \rfloor}
& \ge \exp \Ll(-C N   2^{(\ga-1)k} \Rr) \\
& \ge \exp \Ll(-C N^{1+(\ga-1)(1-\al)} \Rr),
\end{align*}
where we used that $\ga > 1$ and the definition of $k$.
\end{proof}

%
%
%
%
%
%
%
\subsection{Upper bounds}
\label{ss.upper}
In this subsection, we prove upper bounds on the $\P_\ga$-proba\-bility of deviations of the $\ell^1$-average path length $\H_1(\GN)$. Those upper bounds match the lower bounds obtained in Proposition~\ref{p.lower} for the diameter $\H_\infty(\GN)$. 

\begin{proposition}
\label{p.upper.small}
Assume $\ga < 1$. \\
\emph{(1)} For every $\al \in (0,1)$, there exists $c > 0$ such that
$$
\P_\ga[\mcl H_1(\GN) \le N^\al] \le \exp \Ll( -cN^{1-\al(1-\ga)} \Rr).
$$
\emph{(2)} There exists $c > 0$ such that
$$
\P_\ga[\mcl H_1(\GN) \le cN] \le \exp \Ll( -cN^\ga \Rr).
$$ 
\end{proposition}
\begin{proposition}
\label{p.upper.large}
Assume $\ga > 1$. \\
\emph{(1)} For every $\al \in (0,1)$, there exists $c > 0$ such that
$$
\P_\ga \Ll[ \H_1(\GN) \le N^\al \Rr]  \le \exp\Ll(-c N^{1+(1-\al)(\ga-1)} \Rr),
$$
\emph{(2)} There exists $c > 0$ such that
$$
\P_\ga \Ll[ \H_1(\GN) \le c N \Rr]  \le \exp\Ll(-c N \Rr).
$$
\end{proposition}
While part (2) of Propositions~\ref{p.upper.small} and \ref{p.upper.large} are not really needed for the proof of Theorem~\ref{t.main1}, we find it interesting to point out that these small probability estimates already hold as soon as the diameter is required to be a small constant times $N$.

For clarity of exposition, we will prove Proposition~\ref{p.upper.large} first. We start by introducing the notion of $\sigma$-cutpoint, which in its special case $\sigma = 1$ was already used in \cite{benber}. 
For any $\sigma > 0$, we say that $x \in V_N$ is a $\sigma$\emph{-cutpoint} in the graph $\GN$ if no edge $e = \{e^-,e^+\} \in \EN$ is such that $e^- < x$ and $e^+ \ge  x + \si$.  In other words, no edge of length $\si$ passing ``above $x$'' reaches $x + \si$ or further to the right. (In view of the proof of Proposition~\ref{p.lower}, we can anticipate that for $\ga > 1$, we will ultimately choose $\sigma \simeq N^{1-\al}$.) Let $X_0 = 0$, and define recursively
$$
X_{i+1} = \inf \{ x \ge X_i + \si\ : x \mbox{ is a $\sigma$-cutpoint in $\GN$}\},
$$
with the convention that $X_{i+1} = N$ if the set is empty. We also define 
$$
T = \sup \{ i: X_i < N\}.
$$
Both the sequence $(X_i)$ and $T$ depend on $N$ and $\sigma$, although the notation does not make it explicit. The quantity $T$ records a number of $\sigma$-cutpoints that are sufficiently separated from one another. We would like to say that up to a constant, $\H_1(\GN)$ should be at least as large as $T$. 
While this would be correct if $\H_1(\GN)$ was replaced by the diameter $\H_\infty(\GN)$, counter-examples can be produced for $\H_1(\GN)$.  
The next lemma provides us with a suitably weakened version of this idea. There, one should think of $X_{T_1}$ and $(N-X_{T_2})$ as being of order $N$ and of $T_2 - T_1$ as being of order $T$.
\begin{lemma}[average path length via $\si$-cutpoints]
\label{l.T-to-H}
If $0 < T_1 < T_2  \le T$, then
$$
\H_1(\GN) \ge  \frac{2X_{T_1} (N-X_{T_2})}{N^2}  (T_2-T_1).
$$
\end{lemma}
\begin{proof}
Consider the situation where $x, y \in V_N$ and $1 \le j,j'\le T$ are such that
\begin{equation}
\label{e.t-to-h}
x < X_j < X_{j'} \le y.
\end{equation}
Any path connecting $x$ to $y$ must visit each of the intervals $\{X_{i}, \ldots, X_{i+1}-1\}$, where $i \in \{j , \ldots,j'-1\}$. Indeed, it suffices to verify that there is no edge $e = \{e^-, e^+\}$ such that $e^- < X_i$ and $e^+ \ge X_{i+1}$. This is true since $X_i$ is a $\si$-cutpoint and $X_{i+1} - X_i \ge \si$. Hence, if \eqref{e.t-to-h} holds, then $\d_\GN(x,y) \ge j'-j$. As a consequence,
\begin{align*}
\sum_{x,y \in V_N} \d_\GN(x,y) & \ge 2 \sum_{1 \le j < j'\le T} \ \sum_{\substack{ X_{j-1} \le x < X_j \\ X_{j'}\le y < X_{j'+1} }} \d_\GN(x,y) \\
& \ge 2 \sum_{1 \le j < j'\le T} (X_j - X_{j-1})(X_{j'+1} - X_{j'}) (j'-j).
\end{align*}
Restricting the sum to indices such that $1 \le j \le T_1$ and $T_2 \le j' \le T$, we obtain the announced bound.
\end{proof}
In order to proceed with the argument, it is convenient to extend the set of vertices to the full line $\Z$: we consider $\msc E_\infty = \{\{x,y\} \ : \ x \neq y \in \Z\}$, and the random set of edges $\mcl E_\infty$ whose law under $\P_\ga$ is described by
\begin{multline}
\label{e.indep-infty}
\mbox{the events $(\{e \in \mcl E_\infty\})_{e \in \msc E_\infty, |e| > 1}$ are independent}, \\
\mbox{and each event has probability } \exp\Ll(-|e|^\ga \Rr).
\end{multline}
We can and will assume that under $\P_\ga$, the sets $\mcl E_N$ and $\mcl E_\infty$ are coupled so that $\mcl E_N \subset \mcl E_\infty$. In particular, a $\si$-cutpoint in $\mcl G_\infty := (\Z,\mcl E_\infty)$ is a $\si$-cutpoint in $\GN = (V_N,\EN)$. We define the sequence $(\td X_i)_{i \in \N}$ as following the definition of $(X_i)$, but now for the graph $\mcl G_\infty$. That is, we let $\td X_0 = 0$ and for all $i \ge 0$,
$$
\td X_{i+1} := \inf \{ x \ge \td X_i + \si\ : x \mbox{ is a $\sigma$-cutpoint in $\mcl G_\infty$}\}.
$$
The aforementioned coupling guarantees that for every $i \in \N$,
\begin{equation}
\label{e.coupling}
X_i \le \td X_i.
\end{equation}
\begin{lemma}[i.i.d.\ structure]
\label{l.coupling}
The sequence $(\td X_{i+1} - \td X_i)_{i \ge 0}$ is stochastically dominated by a sequence of i.i.d.\ random variables distributed as $\td X_1$. 
\end{lemma}
\begin{proof}
For every $i \ge 0$, the event $\td X_{i+1} - \td X_i > x$ can be rewritten as
$$
\Ll\{ \forall y \in \{\td X_i + \si,\ldots,\td X_i + x\} \ \exists e = \{e^-,e^+\} \in \mcl E_\infty \text{ s.t. } \ e^- < y \text{ and } e^+ \ge y + \si \Rr\}.
$$
For $i \neq 0$, the point $\td X_i$ is a $\si$-cutpoint, hence the event above is not modified if we add the restriction that $e^- \ge \td X_i$. For any given $x_0, \ldots, x_i$, the event 
$$
\{\td X_0 = x_0, \ldots, \td X_i = x_i\}
$$ 
is a function of $(\1_{e \in \mcl E_\infty})$ over edges $e$ whose left endpoint is strictly below~$x_i$. Hence,
\begin{multline*}
\P_\ga \Ll[ \td X_0 = x_0, \ldots, \td X_i = x_i, \, \td X_{i+1} - \td X_i > x \Rr]  \\ 
\ge \P_\ga \Ll[ \td X_0 = x_0, \ldots, \td X_i = x_i\Rr] \, \P_\ga\Ll[\td X_1 > x \Rr],
\end{multline*}
and the lemma is proved.
\end{proof}
\begin{remark}
In fact, the argument above shows that the random variables $(\td X_{i+1} - \td X_i)_{i \ge 1}$ are i.i.d. We could arrange that $(\td X_{i+1} - \td X_i)_{i \ge 0}$ be i.i.d.\ by choosing to define $\mcl G_\infty$ over the vertex set $\N$ instead of $\Z$. However, we prefer to stick to the present setting, which makes the proofs of Lemmas~\ref{l.tail.X.large} and \ref{l.tail.X.small} slightly more convenient to write.
\end{remark}
We now state an estimate on the tail probability of $\td X_1$ in the case $\ga > 1$, and use it to prove Proposition~\ref{p.upper.large}.
\begin{lemma}[Exponential moments of $\td X_1$ for $\ga > 1$]
\label{l.tail.X.large}
For every $\ga > 1$, there exists $c_0 > 0$ and $C_0 < \infty$ (not depending on $\sigma \ge 1$) such that for every $\theta \le c_0 \si^{\ga-1}$,
$$
\E_\ga \Ll[ \exp \Ll( \theta \td X_1 \Rr)  \Rr] \le  \exp \Ll( C_0 \theta \si \Rr) .
$$
\end{lemma}
\begin{proof}
For every $x \in \Z$, we define the \emph{reach} of $x$ in the graph $\mcl G_\infty$ as
\begin{equation}
\label{e.def.R}
R(x) = \sup \{y \ge 0 \ : \ \exists z < 0 \text{ s.t. } \{x+z,x+y\} \in \mcl E_\infty\} \qquad (\ge 0).
\end{equation}
This quantity will be helpful to control $\td X_1$, since  the point $x$ is a $\sigma$-cutpoint if and only if $R(x) < \sigma$; and moreover, the random variables $(R(x))_{x \in \Z}$ are identically distributed. We start by estimating their tail. 
\begin{align*}
\P_\ga[R(0) > r] & \le \sum_{z < 0} \P_\ga[\exists y > r \ : \ \{z,y\} \in \mcl E_\infty] \\
& \le \sum_{z < 0} \sum_{y = r+1}^\infty \exp\Ll(-(y-z)^\ga\Rr) \le C \exp\Ll(-c r^\ga\Rr),
\end{align*}
where the constants $C, c > 0$ depend only on $\gamma$. We can adjust the constant $c > 0$ so that
\begin{equation}
\label{e.tail.R}
\P_\ga[R(0) > r] \le \exp\Ll(-c r^\ga\Rr).
\end{equation}
As a consequence, 
\begin{align*}
\E_\ga \Ll[ \exp \Ll( \theta R(0) \Rr)  \Rr] & \le \exp\Ll(\theta \si\Rr) + \sum_{k = 0}^\infty \exp\Ll(2^{k+1}\theta \si\Rr) \P_\ga \Ll[ 2^k \si < R(0) \le 2^{k+1} \si \Rr]  \\
& \le \exp\Ll(\theta \si\Rr) + \sum_{k = 0}^\infty \exp\Ll(2^{k}\Ll[ 2 \theta \si - c 2^{k(\ga-1)} \si^\ga\Rr]\Rr) .
\end{align*}
Since $\ga > 1$, assuming $\theta \si \le c_1 \si^\ga$ with $c_1>0$ sufficiently small, we have
\begin{equation*}
\E_\ga \Ll[ \exp \Ll( \theta R(0) \Rr)  \Rr] \le  \exp \Ll( 2 \theta \si  + C\Rr) .
\end{equation*}
By Jensen's inequality, for $\theta \le c_1 \si^{\ga-1}$, we can rewrite this estimate in the more convenient form
\begin{align}
\E_\ga \Ll[ \exp \Ll( \theta R(0) \Rr)  \Rr]  & \le \E_\ga \Ll[ \exp \Ll( c_1 \si^{\ga-1}  R(0) \Rr)  \Rr]^{\frac{\theta}{c_1 \si^{\ga-1}}} \notag \\
& \le  \exp \Ll(C_1 \theta \si\Rr) .
\label{e.exp.R}
\end{align}
for some constant $C_1 < \infty$ not depending on $\theta$ or $\sigma$.
We now define inductively $Z_0 = \si$, 
\begin{equation}
\label{e.def.Zi}
Z_{i+1} = Z_i +  R(Z_i),
\end{equation}
and we let
\begin{equation}
\label{e.def.I}
I := \inf\{i \ge 0 :  R(Z_i) \le \si\}.
\end{equation}
The point $Z_i$ is a $\si$-cutpoint if $ R(Z_i) \le \si$, so $\td X_1 \le Z_I$, and we will focus on estimating the exponential moments of $Z_I$. 
By \eqref{e.def.Zi}, no edge $\{e^-,e^+\}$ with  $e^- \le Z_i$ is such that $e^+ > Z_{i+1}$, so
$$
R (Z_{i+1}) = \sup \{e^+ \ge 0 : \exists e^- \in \{Z_{i},\ldots,Z_{i+1}-1\}  \text{ s.t. } \{e^-,Z_{i+1} + e^+\} \in \mcl E_\infty\}.
$$
Conditionally on $R(Z_0), \ldots, R(Z_i)$, the law of the events $(\{e^-, Z_{i+1} + e^+\}\in \mcl E_\infty)$ for $e^-, e^+$ as above are independent, and each has probability $\exp\Ll(-|e|^\ga \Rr)$. Hence, the sequence $(R(Z_i))_{i \in \N}$ is stochastically dominated by a sequence $(R'_i)_{i \in \N}$ of i.i.d.\ random variables distributed as $R(0)$. Letting 
\begin{equation}
\label{e.def.Z'i}
Z'_i =  \si + \sum_{j = 0}^{i-1} R'_j
\end{equation}
and
\begin{equation}
\label{e.def.I'}
I' = \inf \{i \ge 0 : R'_i \le \si\},
\end{equation}
we also have that $Z_I$ is stochastically dominated by $Z'_{I'}$. Our task is thus reduced to evaluating the tail of $Z'_{I'}$. We note that by \eqref{e.tail.R},
\begin{equation}
\label{e.tail.I}
\P_\ga[I' \ge i] = \Ll( \P_\ga \Ll[ R(0) > \si \Rr]  \Rr)^i \le \exp \Ll( -ci\si^\ga \Rr) ,
\end{equation}
and decompose
\begin{align*}
\E_\ga \Ll[ \exp \Ll( \theta Z_I \Rr)  \Rr] & \le \E_\ga \Ll[ \exp \Ll( \theta Z'_{I'} \Rr)  \Rr]  \\
& \le \exp((2^{k_0} + 1)\theta \si) + \sum_{k = k_0}^\infty \exp \Ll( 2^{k+1} \theta \si \Rr) \P_\ga \Ll[ 2^k \si \le Z'_{I'} - \si < 2^{k+1} \si \Rr].
\end{align*}
where $k_0$ is chosen as the smallest integer such that $2^{k_0} \ge 2 C_1$, the constant $C_1$ being that appearing in \eqref{e.exp.R}. 
We have
$$
\P_\ga \Ll[ Z'_{I'} - \si \ge 2^k \si \Rr]  \le \P_\ga [ I' \ge 2^{k-k_0}] + \P_\ga \Ll[ \sum_{j = 0}^{2^{k-k_0}-1} R_j'\ge 2^k\si \Rr]. 
$$
The first term is estimated by \eqref{e.tail.I}. In order to control the second term, we assume that $ \theta  \le \frac{c_1}{8} \si^{\ga-1}$, and use Chebyshev's inequality, independence of the summands and \eqref{e.exp.R} to get
\begin{align*}
\P_\ga \Ll[ \sum_{j = 0}^{2^{k-k_0}-1} R_j'\ge 2^k\si \Rr] & \le \frac{\{\E_\ga[\exp(8\theta R(0))]\}^{2^{k-k_0}}}{\exp(2^{k+3} \theta \si)} \\
& \le  \exp \Ll( 2^{k-k_0+3}C_1 \theta \si - 2^{k+3} \theta \si \Rr)  \\
& \le \exp \Ll( -2^{k+2} \theta \si \Rr) ,
\end{align*}
where we used the definition of $k_0$ in the last step.
We thus obtain, for $\theta  \le \frac{c_1}8 \si^{\ga-1}$, that
\begin{multline*}
\E_\ga \Ll[ \exp \Ll( \theta Z_I \Rr)  \Rr]  \le \exp((2^{k_0} + 1)\theta \si) \\
 + \sum_{k = k_0}^\infty \exp \Ll( 2^{k+1} \theta \si \Rr) \Ll\{ \exp(-c2^{k-k_0} \si^\ga) +  \exp \Ll( -2^{k+2} \theta \si  \Rr) \Rr\} ,
\end{multline*}
and this yields the desired result.
\end{proof}
\begin{proof}[Proof of Proposition~\ref{p.upper.large}]
We begin with part (1) of the proposition. We denote by $c_0$ and $C_0$ the constants appearing in Lemma~\ref{l.tail.X.large}. 
Let $m$ be an integer that will be fixed later in terms of $C_0$ only. 
By Chebyshev's inequality, Lemma~\ref{l.coupling} and Lemma~\ref{l.tail.X.large} with $\theta = c_0 \si^{\ga-1}$, 
\begin{align*}
\P_\ga \Ll[ \td X_{mN^\al} \ge N \Rr] & \le \frac{ \Ll[ \exp \Ll( C_0 c_0 \si^\ga \Rr)  \Rr]^{mN^\al}  }{\exp \Ll( c_0 N \si^{\ga-1} \Rr) } \\
& = \exp \Ll\{- c_0 N^\al \si^\ga  \Ll( \frac{N^{1-\al}}{\si} - C_0 m\Rr)  \Rr\}.
\end{align*}
Fixing $\si =  N^{1-\al}/(2C_0m)$ (which is greater than $1$ for $N$ sufficiently large, since $\al < 1$), we obtain 
$$
\P_\ga \Ll[ \td X_{mN^\al} \ge N \Rr] \le \exp \Ll( -c_1 N^{\al + \ga(1-\al)} \Rr),
$$
for some $c_1 > 0$. 
By \eqref{e.coupling}, on the event $\td X_{mN^\al} < N$, we have $X_{mN^\al} < N$ and thus
$T \ge mN^\al$. On this event, since $X_{i+1} - X_i \ge \si = N^{1-\al}/(2C_0m)$, we also have
$$
X_{mN^\al/3} \ge \frac 1 {6C_0} N \text{ and } N-X_{2mN^\al/3} \ge X_{mN^\al} - X_{2mN^\al/3} \ge \frac 1 {6C_0} N.
$$
By Lemma~\ref{l.T-to-H}, we thus have
$$
\td X_{mN^\al} < N \quad \implies \quad \H(\GN) \ge \frac {2m} {3\cdot (6C_0)^2} N^\al.
$$
Choosing $m = 3 \cdot (6C_0)^2/2$, we obtain
$$
\P_\ga[\H(\GN) \le N^\al] \le \P_\ga[\td X_{mN^\al} \ge N] \le \exp \Ll( -c_1 N^{\al + \ga(1-\al)} \Rr) ,
$$
which proves part (1). 
The proof of part (2) is identical, except that we choose $\si = 1$ throughout. 
\end{proof}

We now turn to the proof of Proposition~\ref{p.upper.small}, that is, we now focus on the case $\ga < 1$. \emph{From now on, we fix $\si = 1$ and call a $1$-cutpoint simply a cutpoint}. If $I$ is an integer interval, we say that a point $x \in I$ is a \emph{local cutpoint in $I$} (for the graph $\GN$) if whenever an edge $e \in \EN$ goes above $x$, none of its endpoints is in $I$, that is,
\begin{equation*}  
\{ e = \{e^-, e^+\} \in \EN \text{ s.t. } e^- < x < e^+ \text{ and } \{e^-, e^+\} \cap I \neq \emptyset\} = \emptyset.
\end{equation*}
We first give a substitute to Lemma~\ref{l.T-to-H} adapted to this notion.

\begin{lemma}[average path length via local cutpoints]
\label{l.local.to.H}
Let $I\subset V_N$ be an integer interval, and $T$ denote the number of local cutpoints in $I$. We have
\begin{equation*}  
\sum_{x,y \in I} \d_\GN(x,y) \ge \frac {T^3}{63}.
\end{equation*}
If $I, I' \subset V_N$ are two disjoint integer intervals, and if $T$ is the minimum between the number of local cutpoints in $I$ and in $I'$, then we also have
\begin{equation*}  
\sum_{x\in I, y \in I'} \d_\GN(x,y) \ge \frac {T^3}{63}.
\end{equation*}
\end{lemma}
\begin{proof}
We only prove the first statement; it will be clear that the proof applies to the second statement as well. 
Let $Y_1 <  \cdots < Y_T$ be an enumeration of the local cutpoints in $I$. Assume that for $1 < j < j' < T$ and $x, y \in I$, we have
\begin{equation*}  
Y_{j-1} \le x < Y_j < Y_{j'} \le y < Y_{j'+1}.
\end{equation*}
As was seen in the proof of Lemma~\ref{l.T-to-H}, if a path joins $x$ to $y$ without exiting $I$, then its length is at least $j'-j$. 

By the definition of $Y_1$, there is no edge linking a point outside of $I$ to a point $x'$ such that $x' > Y_1$. Similarly, there is no edge linking a point $y' < Y_T$ to a point outside of $I$. As a consequence, a path joining $x$ to $y$ faces the following alternative:
\begin{enumerate}
\item go from $x$ to $y$ without exiting $I$;
\item go through a number of excursions to the left of $I$, then reenter $I$ to the left of $Y_1$ and go to $y$ without further exiting $I$;
\item go through a number of excursions to the left of $I$, then jump directly from the left of $I$ to the right of $I$ and do a number of excursions to the right of~$I$, possibly several times jumping back and forth to the left and to the right of $I$, and then finally enter $I$ to the right of $Y_T$ and connect with $y$.
\end{enumerate}
Since we want to find a lower bound on the length of such a path, it suffices to consider the following cases:
\begin{enumerate}
\item the path goes from $x$ to $y$ without exiting $I$;
\item the path first reaches a point $x' \le Y_1$ while staying in $I$, then exits $I$ to its left, then jump to the right of $I$, then reaches $y' \ge Y_T$, and finally reaches~$y$ while staying in $I$. 
\end{enumerate}
We already found the lower bound $j'-j$ for the first scenario. In the second case, the length of the path is at least $(j-1) + 1 + 1 + 1 + (T-j'-1) \ge T-(j'-j)$. Therefore,
\begin{align*}
\sum_{x,y \in I} \d_\GN(x,y) & \ge 2 \sum_{1 \le j < j'\le T} \ \sum_{\substack{ Y_{j-1} \le x < Y_j \\ Y_{j'}\le y < Y_{j'+1} }} \d_\GN(x,y) \\
& \ge 2 \sum_{1 \le j < j'\le T} (Y_j - Y_{j-1})(Y_{j'+1} - Y_{j'}) [(j'-j)\wedge(T-j'+j)].
\end{align*}
Restricting the sum to indices such that 
\begin{equation*}  
\frac T 5 \le j \le \frac {2T} 5 \ \ \text{ and } \ \ \frac {3T}5 \le j' \le \frac{4T}{5},
\end{equation*}
and observing that $Y_{2T/5} - Y_{T/5} \ge T/5$ and $Y_{4T/5} - Y_{3T/5} \ge T/5$, we obtain the result.
\end{proof}

We now estimate the tail probability of $\td X_1$ (recall that we fixed $\si = 1$).
\begin{lemma}[Exponential moment of $\td X_1$ for $\ga < 1$] 
\label{l.tail.X.small}
For every $\ga < 1$, there exists $\theta > 0$ such that 
$$
\E_\ga \Ll[ \exp \Ll( \theta \td X_1^\ga \Rr)  \Rr] < \infty.
$$
\end{lemma}
\begin{proof}
We first recall some elements of the proof of Lemma~\ref{l.tail.X.large}. 
We define $R(x)$ as in \eqref{e.def.R}, and observe that the estimate \eqref{e.tail.R} still holds under our present assumption $\ga < 1$. We also define $(Z_i)$ and $I$ as in \eqref{e.def.Zi} and \eqref{e.def.I} respectively (with $\si = 1$). We have that $\td X_1 \le Z_I$, and that the sequence $(R(Z_i))$ is stochastically dominated by a sequence $(R'_i)_{i \in \N}$ of i.i.d.\ random variables distributed as $R(0)$. We define $(Z'_i)$ by \eqref{e.def.Z'i} and $I'$ by \eqref{e.def.I'}, and recall that $Z_I$ is stochastically dominated by $Z'_{I'}$. 

As in Lemma~\ref{l.tail.X.large}, our final goal is to estimate the exponential moments of $Z'_{I'}$. We start by estimating those of $R(x)$:
\begin{align*}
\E_\ga \Ll[ \exp \Ll( \theta R(0)^\ga \Rr)  \Rr] & \le \exp\Ll(\theta \Rr) + \sum_{k = 0}^\infty \exp\Ll(\theta2^{\ga(k+1)} \Rr) \P_\ga \Ll[ 2^k  < R(0) \le 2^{k+1}  \Rr]  \\
& \le \exp\Ll(\theta \Rr) + \sum_{k = 0}^\infty \exp \Ll[-2^{\ga k} \Ll( c-2^{\ga} \theta \Rr) \Rr] .
\end{align*}
For $\theta >0$ sufficiently small, we thus have 
$$
\E_\ga \Ll[ \exp\Ll(\theta R(0)^\ga \Rr) \Rr]  < \infty.
$$
By Proposition~\ref{p.ldp-stretch} of the Appendix, letting $C_0 := \E_\ga[R(0)] + 1$, there exists $c_0 > 0$ such that 
\begin{equation}
\label{e.bound.sum}
\P_\ga \Ll[ \sum_{j = 0}^{i-1} R'_j \ge C_0 i \Rr] \le \exp(-c_0 i^\ga).
\end{equation}
Recall from \eqref{e.tail.R} that
$$
\P_\ga[R(0) > 1] \le \exp(-c) < 1,
$$
and thus
\begin{equation}
\label{e.bound.iprime}
\P_\ga[I'\ge i] \le \exp(-ci).
\end{equation}
We now write
$$
\E_\ga \Ll\{ \exp \Ll[ \theta (Z'_{I'})^\ga \Rr]  \Rr\} \le \exp(\theta) + \sum_{k = 0}^\infty \exp\Ll(\theta 2^{\ga(k+1)} \Rr) \P_\ga\Ll[2^k \le Z'_{I'} - 1 < 2^{k+1} \Rr],
$$
and bound the probability on the right-hand side by
$$
\P_\ga \Ll[ Z'_{I'} -1 \ge 2^{k} \Rr] \le \P_\ga[I' > i] + \P_\ga \Ll[ \sum_{j = 0}^{i-1} R'_j \ge 2^{k}  \Rr].
$$
The estimate above is valid for every $i$. We choose $i = 2^k/C_0$, so that the second term on the \rhs\ is bounded by \eqref{e.bound.sum}. Using \eqref{e.bound.iprime} on the first term, we obtain
\begin{multline*}
\E_\ga \Ll\{ \exp \Ll[ \theta (Z'_{I'})^\ga \Rr]  \Rr\} \\
\le \exp(\theta) + \sum_{k = 0}^\infty \exp\Ll(\theta 2^{\ga(k+1)} \Rr) \Ll[ \exp \Ll( -c \frac { 2^k}{C_0} \Rr) + \exp \Ll( -c_0    \frac{2^{\ga k}}{C_0^\ga} \Rr)   \Rr] ,
\end{multline*}
and the latter series is finite when $\theta > 0$ is sufficiently small.
\end{proof}
\begin{cor}
\label{c.cutpoints}
For every $\ga < 1$, there exists $c_1 > 0$ such that
$$
\P_\ga \Ll[ \Ll| \Ll\{ x \in \{0,\ldots,N-1\} \ : \ x \mbox{ is a cutpoint in $\mcl G_\infty$} \Rr\}  \Rr| < c_1 N \Rr] \le \exp \Ll( -c_1 N^\ga \Rr) .
$$
In particular,
$$
\P_\ga\Ll[\GN \mbox{ has less than $ c_1N$ cutpoints}\Rr] \le \exp \Ll( -c_1 N^\ga \Rr) .
$$
\end{cor}
\begin{proof}
In order to prove the corollary, it suffices to see that for some $c > 0$ sufficiently small,
$$
\P_\ga \Ll[ \td X_{cN} \ge N \Rr] \le \exp \Ll( -c N^\ga \Rr) .
$$
This is a consequence of Lemmas~\ref{l.coupling}, \ref{l.tail.X.small} and Proposition~\ref{p.ldp-stretch}.
\end{proof}
We are now ready to complete the proof of Proposition~\ref{p.upper.small}. In this proof, we will consider integer intervals $I \subset V_N$, and discuss the notion of being a cutpoint in the graph induced by the vertex set $I$. Before going to the details, we wish to emphasize that this notion is defined only in terms of edges with $\emph{both}$ endpoints in $I$. It is therefore different from the notion of being a local cutpoint in $I$ (for the graph~$\GN$), since in the latter case, every edge having \emph{at least one} endpoint in $I$ matters.
\begin{proof}[Proof of Proposition~\ref{p.upper.small}]
We fix $c_1 > 0$ as in Corollary~\ref{c.cutpoints}. Note that since we fixed $\si = 1$, the sequence $(X_i)_{i \ge 1}$ is just enumerating the sequence of cutpoints. By Lemma~\ref{l.local.to.H}, we have
\begin{equation}
\label{e.cut-imp-H}
\GN \mbox{ has at least $c_1 N$ cutpoints} \implies \H(\GN) \ge  \frac{c_1^3}{63}  \ N. 
\end{equation}
Hence, part (2) of the proposition is a consequence of Corollary~\ref{c.cutpoints}.

We now turn to part (1). 
Throughout the argument, we denote by $c > 0$ a generic constant whose value may change from place to place to be as small as necessary, and is not allowed to depend on $N$.

We partition $V_N$ into $K := N^{1-\al}/3$ integer intervals of length $3 N^\al$, which we denote by $I_1, \ldots, I_K$. For each $k \in \{1,\ldots,K\}$, we denote by $J_k$ the middle third interval in $I_k$.  Let $\msc C_k$ be the set of cutpoints induced by the vertex set $I_k$, and let $\mcl C_k$ denote the event that
\begin{equation*}  
| J_k \cap \msc C_k  |  \ge  c_1  N^\al.
\end{equation*}
By construction, the events $(\mcl C_k)_{1 \le k \le K}$ are independent. Moreover, each has probability at least 
$
1-\exp\Ll(-c_1  N^{\al \ga}\Rr),
$ 
by Corollary~\ref{c.cutpoints}. Consequently, the probability that 
\begin{equation}
\label{e.good.event}
\Ll| \Ll\{ k \in \{1,\ldots,K\} \ : \ \mcl C_k \text{ holds} \Rr\}   \Rr| \ge \frac K 2
\end{equation}
is at least
\begin{equation*}  
1 - \exp \Ll( -c N^{1-\al+\al\ga} \Rr) ,
\end{equation*}
by a standard calculation (see e.g.\ \cite[(2.15)-(2.16)]{lyap}). We may therefore assume that the event \eqref{e.good.event} holds.

Let $\mcl B$ denote the event
\begin{equation*}
\Ll| \Ll\{ e \in \EN \ : \ |e| \ge  N^\al \Rr\}  \Rr| \le \frac{N^{1-\al}}{20}.
\end{equation*}
We now argue that for some $c > 0$,
\begin{equation}
\label{e.prob.B}
\P_\ga[\mcl B] \ge 1 - \exp \Ll( -c N^{1-\al(1-\ga)} \Rr) .
\end{equation}
In order to do so, we use independence to note that there exists a constant $C < \infty$ such that for every $\lambda \in \Ll[0,\frac 1 2 N^{\al \gamma}\Rr] $, we have
\begin{equation*}  
\E \Ll[ \exp \Ll( \lambda \sum_{|e| \ge N^\al} \1_{e \in \EN} \Rr)  \Rr] \le \Ll(1 + \exp \Ll( \lambda - N^{\al \gamma}\Rr) \Rr)^{N^2} \le C,
\end{equation*}
and therefore, by Chebyshev's inequality,
\begin{equation*}  
1 - \P_\ga[\mcl B] \le C \exp \Ll( - \frac {N^{\al \gamma}} 2 \, \frac{N^{1-\al}}{20} \Rr) ,
\end{equation*}
so that \eqref{e.prob.B} is proved.


From now on, we therefore assume that both the event $\mcl B$ and the event in \eqref{e.good.event} are realized, and show that this implies $\H(\GN) \ge c N^\al$.

Denote the set of endpoints of edges with length at least $ N^\al$ by
$$
\msf{End}_N := \{x \in V_N \ : \ \exists y \text{ s.t. } \{x,y\} \in \EN \text{ and } |y-x| \ge  N^\al\}.
$$
Since we assume the event $\mcl B$ to be realized, the set $\msf{End}_N$ contains no more than $ N^{1-\al}/10$ points. Since we also assume \eqref{e.good.event}, we can isolate at least 
$$
K' := \frac K 2 - \frac 1 {10} N^{1-\al} = \Ll(\frac 1 {6 } - \frac 1 {10}\Rr) N^{1-\al} 
$$
pairwise disjoint intervals $I_{l_1}, \ldots, I_{l_{K'}}$ such that for every $k \in \{1,\ldots,K'\}$,
$$
I_{l_k} \cap \msf{End}_N = \emptyset \ \ \text{ and } \ \ |J_{l_k} \cap \msc C_{l_k}| \ge {c_1 } N^\al.
$$
Fix $k \in \{1,\ldots, K'\}$. We now show that 
\begin{equation}
\label{e.local.cut}
\mbox{there are at least ${c_1 } N^\al$ \emph{local} cutpoints in $I_{l_k}$.} 
\end{equation}
As recalled before the beginning of the proof, the potentially problematic edges are those with one endpoint in $I$ and one outside of $I$. Since $I_{l_k}$ contains no element of $\msf{End}_N$, no such edge can have length larger than $ N^\al$. Therefore, if a point is at distance at least $ N^\al$ from the extremities of $I_{l_k}$, then there is no edge going above it and that has exactly one endpoint outside of $I_{l_k}$. Since we chose $J_{l_k}$ as the middle third interval in $I_{l_k}$, and $I_{l_k}$ is of total length $3 N^\al$, this yields \eqref{e.local.cut}.

By Lemma~\ref{l.local.to.H}, we deduce that for every $k,k' \in \{1,\ldots,K'\}$, we have
\begin{equation*}  
\sum_{x \in I_{l_k},y \in I_{l_{k'}}} \d_\GN(x,y) \ge c N^{3\al}.
\end{equation*}
Summing over $k,k'$ and recalling that $K' \ge c N^{1-\al}$, we obtain that
$\mcl H_1(\GN) \ge c N^\al$, as desired.
\end{proof}

\subsection{Conclusion}
In this final subsection, we complete the proof of Theorem~\ref{t.main1}.
\begin{proof}[Proof of Theorem~\ref{t.main1}]
Fix $\ga < 1$, $p \in [1,\infty]$, $b \in (\gamma - 1, 1)$ and 
\begin{equation*}  
\al := \frac {1-b}{2-\ga} \in (0,1). 
\end{equation*}
Let $\eps > 0$ be sufficiently small, and let $\al' \in (0,1) \setminus (\al - 2\eps,\al + 2\eps)$. By the comparisons $\H_1 \le \H_p \le \H_\infty$ and Propositions~\ref{p.lower} and \ref{p.upper.small}, there exists a constant $C < \infty$ such that 
\begin{multline}  \label{e.ratio1}
\frac{\Pb[N^{\al-\eps} \le \H_p(\GN) \le N^{\al + \eps}]}{\Pb[N^{\al' - \eps} \le \H_p(\GN) \le N^{\al'+\eps}]} \\
\ge \frac{\exp \Ll(- C^{-1} \Ll[ N^{b+\al+\eps} + N^{1-(\al-\eps)(1-\ga)} \Rr] \Rr)}{\exp \Ll( -C \Ll[N^{b+\al'-\eps} + N^{1-(\al'+\eps)(1-\ga)} \Rr]\Rr)},
\end{multline}
The function
\begin{equation*}  
\td \al \mapsto (b+\td \al) \vee (1-\td \al(1-\ga))
\end{equation*}
attains a strict minimum at the value $\td \al = \al$. Reducing $\eps > 0$ as necessary, we can make sure that the right-hand side of \eqref{e.ratio1} tends to infinity as $N$ tends to infinity. The other cases are handled similarly. For example, when $\ga > 1$ and $b \in (0,\ga)$, we fix
\begin{equation*}  
\al := \frac{\ga - b}{\ga} \in (0,1),
\end{equation*}
take $\al' \in (0,1) \setminus (\al - 2\eps,\al + 2\eps)$, and observe that 
\begin{multline*}  
\frac{\Pb[N^{\al-\eps} \le \H_p(\GN) \le N^{\al + \eps}]}{\Pb[N^{\al' - \eps} \le \H_p(\GN) \le N^{\al'+\eps}]} \\
\ge \frac{\exp \Ll(- C^{-1} \Ll[ N^{b+\al+\eps} + N^{1+(1-\al - \eps)(\ga-1)} \Rr] \Rr)}{\exp \Ll( -C \Ll[N^{b+\al'-\eps} + N^{1+(1-\al' + \eps)(\ga-1)} \Rr]\Rr)}.
\end{multline*}
The exponent $\al$ was chosen to be realize the strict minimum of the function
\begin{equation*}  
\td \al \mapsto (b+ \td \al) \vee (1 + (1-\td \al)(\ga -1)),
\end{equation*}
so the conclusion follows as before.
\end{proof}

%
%
%
%
%
%
%
\section{Critical case}
\label{s.gamma1}

The goal of this section is to prove Theorem~\ref{t.main2}. The main step of the proof consists in showing the following upper and lower bounds on the probability of deviations of the average path length $\H_p(\GN)$ under the measure $\P_1$.

\begin{proposition}
\label{p.crit}
$\mathrm{(i.)}$ For any $p \in [1,\infty]$, $k \in \N$ and $N$ large enough, we have
\begin{align}\label{eq:want_dir1_final}
&\P_1\left[\mathcal{H}_p(\mathcal{G}_N) \leq 3kN^{\frac{1}{k}}\right] \geq \exp\left\{-(k-1)N\right\}.\end{align}
$\mathrm{(ii.)}$ Assume  $p \in [1,\infty]$, $k \in \mathbb{N}$, $\eta \in \left(\frac{1}{k+1},\frac{1}{k}\right)$ and 
\begin{equation}\label{eq:def_zeta_p}\zeta < \zeta_p(\eta):= \Ll| \begin{array}{ll}\frac{p}{k+2p}(1-k\eta)&\text{if } p \in [1,\infty),\\[.2cm]\frac{1}{2}(1-k\eta)&\text{if } p = \infty. \end{array}\Rr.\end{equation}
 Then, for $N$ large enough we have
\begin{equation}\label{eq:want_dir2_final}\P_1\left[\mathcal{H}_p(\mathcal{G}_N) \leq N^\eta\right] \leq \exp\left\{-k N + N^{1-\zeta}\right\}.\end{equation}
\end{proposition}
The proof of this proposition rests on the following two lemmas, which involve no probability. For each $g \in \msc G_N$, we denote
\begin{equation*}  
\mcl C(g) := \mcl C_1(g) = \sum_{e \in E_N \atop |e| > 1} |e|.
\end{equation*}
\begin{lemma}\label{lem:enough_construc}
For any $k \in \N$ and $N$ large enough, there exists $g \in \msc G_N$ such that $\C(g) \leq (k-1)N$ and $\H_\infty(g) \leq 3kN^\frac{1}{k}$.
\end{lemma}
\begin{lemma}\label{lem:cost_is_high}
Let $p \in [1,\infty]$, $k \in \N$, $\eta \in \left(\frac{1}{k+1},\frac{1}{k}\right)$ and $\delta \in (0,1-k \eta)$. For every $N$ large enough and $g = (V_N,E_N) \in \msc G_N$, we have the implication
\begin{equation}\mathcal{H}_p(g) \leq N^\eta
\quad \implies \quad 
\sum_{\substack{e \in E_N:\\ |e| \geq N^\delta}} |e| \geq k N - N^{1-\zeta_{p,\delta}(\eta) }\cdot (\log N)^{6k},
\label{eq:new_want_final}\end{equation}
where
\begin{equation}\label{eq:def_zeta_delta}
\zeta_{p,\delta}(\eta) = \Ll|\begin{array}{ll} \frac{p}{k+p}(1-k \eta - \delta)&\text{if } p \in [1,\infty),\\[.2cm] 1-k \eta - \delta,&\text{if } p = \infty.\end{array}\Rr.\end{equation}
\end{lemma}

In Subsection \ref{ss:proof_of_prop_and_thm}, we show how Lemmas \ref{lem:enough_construc} and \ref{lem:cost_is_high} imply Proposition \ref{p.crit}, and how this proposition in turn gives Theorem \ref{t.main2}. In Subsection \ref{ss:proof_deterministic}, we prove the two lemmas.

\subsection{Proofs of Proposition \ref{p.crit} and Theorem \ref{t.main2}}
\label{ss:proof_of_prop_and_thm}
\begin{proof}[Proof of Proposition~\ref{p.crit}]
For the first statement, let $\bar E_N$ be the set of edges in a graph as described in Lemma \ref{lem:enough_construc}. The desired result follows from
\begin{align*}\P_1\left[\mathcal{H}_p(\mathcal{G}_N) \leq 3kN^{\frac{1}{k}}\right]&\geq  \P_1\left[\mathcal{H}_\infty(\mathcal{G}_N)\leq 3kN^{\frac{1}{k}}\right] \\&\geq \prod_{e \in \bar E_N} \exp\{-|e|\} \geq \exp\left\{-(k-1)N\right\}.
\end{align*}

We now turn to the second statement. Fix $p \in [1,\infty]$, $k \in \N$ and  $\eta \in \left(\frac{1}{k+1},\frac{1}{k}\right)$. Also let $\delta \in (0, 1-k \eta)$ to be chosen later. For any $\theta > 0$ we have
\begin{align}\nonumber
\mathbb{E}_1\left[\exp\left\{ \theta \cdot \sum_{\substack{e\in E_N:|e|\geq N^\delta}} |e|\right\} \right]&=\prod_{\substack{e \in \msc{E}_N:\\|e|\geq N^\delta}} \mathbb{E}_1\left[\exp\left\{\theta \cdot |e| \cdot  \mathds{1}_{\{e\in E_N\}}\right\}\right]\\
\nonumber&\leq \prod_{\substack{e \in \msc{E}_N:\\|e|\geq N^\delta}} \left(1+\exp\{(\theta -1) |e|\}\right)\\
\nonumber&\leq \prod_{i = \lfloor N^\delta \rfloor}^N \;\prod_{\substack{e \in \msc{E}_N:\\|e|=i}} \exp\left\{\exp \left\{(\theta -1) i \right\} \right\}\\
&\leq \exp \left\{N \sum_{i=\lfloor N^\delta \rfloor}^\infty \exp\left\{(\theta -1) i  \right\} \right\}. \label{eq:sum_i_N_floor}
\end{align}
If $\delta'< \delta$ and $\theta = 1 - N^{-\delta'}$, \eqref{eq:sum_i_N_floor} implies that, for $N$ large enough,
\begin{equation}\label{eq:less_than_2}
\mathbb{E}_1\left[\exp\left\{ \theta \cdot \sum_{\substack{e\in E_N:|e|\geq N^\delta}} |e|\right\} \right] \leq 2.
\end{equation}
Then, using Lemma \ref{lem:cost_is_high} and Chebyshev's inequality, if $N$ is large enough,
\begin{align}\nonumber\mathbb{P}_1\left[\mathcal{H}_p(\GN) \leq N^\eta\right] 
& \leq \P_1\left[\sum_{\substack{e \in E_N:\\|e| \ge N^\delta}}|e| \geq k N - (\log N)^{6k}N^{1-\zeta_{p,\delta}(\eta)} \right]
\\\nonumber&\stackrel{\eqref{eq:less_than_2}}{\leq} 2\exp\left\{-(1-N^{-\delta'})\left(k N -(\log N)^{6k}N^{1-\zeta_{p,\delta}(\eta)}\right) \right\} \\&\hspace{.15cm}\leq 2 \exp\left\{-k N + (\log N)^{6k}N^{1-\zeta_{p,\delta}(\eta)} + k N^{1-\delta'}\right\}.\label{eq:right_hand_optimal}\end{align}
We are still free to choose $\delta$ and $\delta'< \delta$. Having in mind the two exponents of $N$ that appear in \eqref{eq:right_hand_optimal}, we choose $\delta$ solving
$$1-\delta = 1-\zeta_{p,\delta}(\eta);$$
this is achieved for $\delta = \zeta_p(\eta)$, as defined in \eqref{eq:def_zeta_p}. Next, we take $\zeta < \zeta_p(\eta)$, as in the statement of the proposition. Observing that $\delta = \zeta_{p,\delta}(\eta) = \zeta_p(\eta)$, we can choose $\delta'$ so that
$$1-\zeta_{p,\delta}(\eta) = 1-\delta < 1-\delta' < 1-\zeta.$$ Then, for $N$ large enough the expression in \eqref{eq:right_hand_optimal} is smaller than $\exp\{-k N + N^{1-\zeta}\}$ as required. 
\end{proof}

\begin{proof}[Proof of Theorem~\ref{t.main2}]
Define
$$A_{k,\varepsilon, N } = \left[N^{\frac{1}{k+1} - \varepsilon}, \;N^{\frac{1}{k+1} + \varepsilon}\right],\qquad k,N\in\N,\;\varepsilon > 0.$$
The desired statement will follow from proving that, for any $k \in \N$, if $\varepsilon > 0$ is small enough and
\begin{equation}
\label{eq:where_is_b} b \in \left(\frac{k-1}{k} + h(k,p)+2\varepsilon,\; \frac{k}{k+1} - 2\varepsilon\right),
\end{equation}
then
\begin{equation}\label{eq:after_comparing}
\P_1^{b,p}\left[\mathcal{H}_p(\mathcal{G}_N) \in A_{k,\varepsilon, N}\right] \xrightarrow{N \to \infty} 1.
\end{equation}
To this end, recalling the definition of $Z^{b,p}_{\ga,N}$ in \eqref{e.def.Zbp}, we start bounding:
\begin{align}\nonumber
&Z^{b,p}_{1,N} \, \P_1^{b,p}\left[\mathcal{H}_p(\mathcal{G}_N) \in A_{k,\varepsilon, N}\right] = \E_1\left[\exp\left\{-N^b\cdot  \mathcal{H}_p(\mathcal{G}_N) \right\} \cdot \mathds{1}\left\{\mathcal{H}_p(\mathcal{G}_N) \in A_{k,\eps,N} \right\}\right] \\[.2cm]
\nonumber& \geq \E_1\left[\exp\left\{-N^b\cdot  \mathcal{H}_p(\mathcal{G}_N) \right\} \cdot \mathds{1}\left\{N ^{\frac{1}{k+1}-\eps} \leq \mathcal{H}_p(\mathcal{G}_N) \leq 3(k+1)N^{\frac{1}{k+1}} \right\}\right] \\[.2cm]
\nonumber&\geq \exp\left\{-3(k+1)N^{b+\frac{1}{k+1}}\right\}  \cdot \left(\P_1\left[\mathcal{H}_p(\mathcal{G}_N) \leq 3(k+1)N^{\frac{1}{k+1}}\right] - \P_1\left[\mathcal{H}_p(\mathcal{G}_N) \leq N^{\frac{1}{k+1}-\varepsilon}\right]\right)\\[.2cm]
\nonumber&\stackrel{\eqref{eq:want_dir1_final}, \eqref{eq:want_dir2_final}}{\geq} \exp\left\{-3(k+1)N^{b+\frac{1}{k+1}}\right\} \cdot \left(\exp\{-kN\} - \exp\{-(k+1) N + o_\varepsilon(N) \}\right),
\end{align}
where $o_\varepsilon(N)$ is a function that depends on $k$, $\varepsilon$ and $N$ and satisfies $o_\varepsilon(N)/N \to 0$ as $N \to \infty$. We thus obtain
\begin{equation}\label{eq:lower_bound_to_compare}
Z^{b,p}_{1,N} \, \P_1^{b,p}\left[\mathcal{H}_p(\mathcal{G}_N) \in A_{k,\varepsilon, N}\right] \geq \frac12\exp\left\{-kN-3(k+1)N^{b+\frac{1}{k+1} }  \right\}.
\end{equation}
We note that, by \eqref{eq:where_is_b}, we have $b + \frac{1}{k+1} < 1$, so
\begin{equation}
\label{eq:much_much_less} N^{b + \frac{1}{k+1} } \ll N\quad \text{as } N \to \infty,
\end{equation}
hence the term $-3(k+1)N^{b + \frac{1}{k+1} }$ is negligible (in absolute value) compared to $-kN$ in the exponential on the right-hand side of \eqref{eq:lower_bound_to_compare}.

Now that we have this lower bound, let us explain how the rest of the proof will go. Define
$$
A^{(0)}_{k,\varepsilon,N} = \left[0,\;N^{\frac{1}{k+1}-\varepsilon}\right],\quad A^{(1)}_{k,\varepsilon,N} = \left[N^{\frac{1}{k+1}+\varepsilon},\;N^{\frac{1}{k}-\varepsilon}\right], \quad A^{(2)}_{k,\varepsilon,N} = \left[N^{\frac{1}{k}-\varepsilon},\;N\right],
$$
so that $[0, N] = A_{k,\varepsilon,N} \cup A^{(0)}_{k,\varepsilon,N} \cup A^{(1)}_{k,\varepsilon,N} \cup A^{(2)}_{k,\varepsilon,N}$. We will obtain upper bounds for
$$Z^{b,p}_{1,N} \, \P_1^{b,p}\left[ \mathcal{H}_p(\mathcal{G}_N) \in A^{(i)}_{k,\eps,N}\right],\quad i\in\{0,1,2\}$$
that will all be negligible compared to the right-hand side of \eqref{eq:lower_bound_to_compare} as $N \to \infty$. From this, \eqref{eq:after_comparing} will immediately follow.

\begin{itemize}\item[(a)]\textbf{Upper bound for $\P_1^{b,p}\left[ \mathcal{H}_p(\mathcal{G}_N) \in A^{(0)}_{k,\eps,N}\right]$}\\
This bound is quite simple:
\begin{align*}\nonumber
&Z^{b,p}_{1,N} \, \P_1^{b,p}\left[\mathcal{H}_p(\mathcal{G}_N) \in A^{(0)}_{k,\varepsilon, N}\right] \\[.2cm]&\qquad= \E_1\left[\exp\left\{-N^b\cdot  \mathcal{H}_p(\mathcal{G}_N) \right\} \cdot \mathds{1}\left\{\mathcal{H}_p(\mathcal{G}_N) \in A^{(0)}_{k,\eps,N} \right\}\right] \\[.2cm]\nonumber&\qquad\leq \P_1\left[\mathcal{H}_p(\mathcal{G}_N) \in A^{(0)}_{k,\eps,N}  \right]\\[.2cm]
&\qquad\stackrel{\eqref{eq:want_dir2_final}}{\leq} \exp\left\{-(k+1)N - o_\eps(N) \right\}.
\end{align*}
Using \eqref{eq:much_much_less}, it is then readily seen that the right-hand side above is negligible compared to the right-hand side of \eqref{eq:lower_bound_to_compare}.
\item[(b)] \textbf{Upper bound for $\P_1^{b,p}\left[ \mathcal{H}_p(\mathcal{G}_N) \in A^{(2)}_{k,\eps,N}\right]$}\\
Similarly to the previous bound,
\begin{align*}\nonumber
Z^{b,p}_{1,N} \, \P_1^{b,p}\left[\mathcal{H}_p(\mathcal{G}_N) \in A^{(2)}_{k,\varepsilon, N}\right] &\leq \exp\left\{-N^{b+ \frac{1}{k}-\eps} \right\}\cdot \P_1\left[ \mathcal{H}_p(\mathcal{G}_N) \in A^{(2)}_{k,\eps,N}\right] \\[.2cm]
&\leq \exp\left\{-N^{b+ \frac{1}{k}-\eps} \right\}.
\end{align*}
In order to show that this is negligible compared to the right-hand side of \eqref{eq:lower_bound_to_compare}, we note that, due to \eqref{eq:where_is_b}, we have
$$N^{b+ \frac{1}{k}-\eps} \gg kN +3(k+1)N^{b + \frac{1}{k+1} }\qquad \text{ as } N \to \infty. $$

\item[(c)] \textbf{Upper bound for $\P_1^{b,p}\left[ \mathcal{H}_p(\mathcal{G}_N) \in A^{(1)}_{k,\eps,N}\right]$}\\
This bound is harder than the previous two, as in this case it is not enough to dismiss the term $N^{1-\zeta}$ in \eqref{eq:want_dir2_final} as being $o(N)$. Rather, in the comparison with \eqref{eq:lower_bound_to_compare}, this term is now decisive. This complication is what leads to the introduction of the function $h(k,p)$ in \eqref{eq:def_of_h} (and the corresponding dark parts of Figure \ref{fig:critical}).

We define $f, g: [\frac{1}{k+1},\frac{1}{k}] \to \mathbb{R}$ by
\begin{align*}
&f(\eta) = b + \eta,\\[.2cm] &g(\eta) = \Ll|\begin{array}{ll}\frac{k + p + kp\eta}{k+2p}&\text{if } p \in [1,\infty),\\[.2cm] \frac12 + \frac{k}{2}\eta&\text{if } p = \infty.  \end{array} \Rr.
\end{align*}
The definition of $g$ is motivated by the fact that
\begin{equation}
\label{eq:why_g}
1-g(\eta) = \zeta_p(\eta) \quad \text{for all } \eta \in \left(\frac{1}{k+1},\frac{1}{k}\right),
\end{equation}
where $\zeta_p(\eta)$ was defined in \eqref{eq:def_zeta_p}. We also note that the function $h(k,p)$ defined in \eqref{eq:def_of_h} satisfies
\begin{equation}
\label{eq:why_h}h(k,p) = \left(g\left(\frac{1}{k+1}\right) - \frac{k-1}{k} - \frac{1}{k+1} \right)\vee 0
\end{equation}

We now claim that $f(\eta) > g(\eta)$ for all $\eta \in \left[\frac{1}{k+1},\frac{1}{k}\right]$. Indeed, since both $f$ and $g$ are affine functions of $\eta$, this follows from
\begin{align*}
&f\left(\frac{1}{k}\right) \stackrel{\eqref{eq:where_is_b}}{>} \frac{k-1}{k} + h(k,p) + \frac{1}{k} \geq 1 = g\left(\frac{1}{k}\right),\\[.2cm]
&f\left(\frac{1}{k+1}\right) \stackrel{\eqref{eq:where_is_b}}{>} \frac{k-1}{k} + h(k,p) + \frac{1}{k+1} \stackrel{\eqref{eq:why_h}}{\geq} g\left(\frac{1}{k+1}\right).
\end{align*}
As a consequence, we can find $\varepsilon' > 0$ and a partition of the interval $\left[\frac{1}{k+1}+\varepsilon,\frac{1}{k}-\varepsilon\right]$ with numbers $\eta_0 = \frac{1}{k+1} + \varepsilon < \eta_1 < \cdots < \eta_r = \frac{1}{k} - \varepsilon$ such that
\begin{equation}
f(\eta_i) > g(\eta_{i+1}) + \varepsilon' \quad \text{ for all } i.
\label{eq:compare_f_g}\end{equation}

We now have
\begin{align*}
&Z^{b,p}_{1,N} \, \P^{b,p}_1\left[\mathcal{H}_p(\mathcal{G}_N) \in A^{(1)}_{k,\eps,N}\right] \\&\qquad\leq \sum_i Z^{b,p}_{1,N} \, \P^{b,p}_1\left[N^{\eta_i} \leq  \mathcal{H}_p(\mathcal{G}_N) \leq N^{\eta_{i+1}} \right]\\[.2cm]
&\qquad\leq \sum_i \exp\{-N^{b+\eta_i}\} \cdot \P_1\left[ \mathcal{H}_p(\mathcal{G}_N) \leq N^{\eta_{i+1}} \right]\\[.2cm]&\qquad\stackrel{\eqref{eq:want_dir2_final},\eqref{eq:why_g}}{\leq} \sum_i \exp\{-N^{b+\eta_i}-kN + N^{g(\eta_{i+1})+\varepsilon'}\}.
\end{align*}
In order to show that each of the terms of the above sum is negligible compared to the right-hand side of \eqref{eq:lower_bound_to_compare}, we need to check that, for all $i$, 
$$ N^{b+\eta_i}  \gg N^{g(\eta_{i+1})+\varepsilon'} +3(k+1)N^{b+\frac{1}{k+1}} \qquad \text{ as } N \to \infty. $$
But this follows promptly from \eqref{eq:compare_f_g} and the fact that $\eta_i > \frac{1}{k+1}$ for each $i$, so we are done.

\end{itemize}
\end{proof}

\subsection{Proof of deterministic lemmas}\label{ss:proof_deterministic}
\begin{proof}[Proof of Lemma~\ref{lem:enough_construc}] Let $L = \lfloor N^{\frac{1}{k}}\rfloor$ and $z_{i,j} = iL^j$, for $i,j$ with $j \in \{1,\ldots, k-1\}$ and $i\in \{0,\ldots, \lfloor (N-1)/L^j \rfloor\}.$ Then define $E_N$ as the set of edges in $E^\circ_N$ together with all edges of the form $\{z_{i,j}, z_{i+1,j}\}$, and let $g = (V_N, E_N)$. We clearly have $\C(g)\leq (k-1)N$. Moreover, writing $S_0 = V_N$ and $S_j = \cup_i\{z_{i,j}\}$ for $j \in \{1,\ldots, k-1\}$, we have
\begin{align*}&\d_g(x,S_{j+1}) \leq L \text{ for all } x\in S_j \text{ and } j\in \{0,\ldots, k-2\};\\
&\d_g(x,y) \leq \frac{N}{L^{k-1}} = \frac{N}{\lfloor N^{1/k}\rfloor^{k-1}} \leq 2N^\frac{1}{k} \text{ for all } x,y \in S_{k-1}
\end{align*}
if $N$ is large enough; from this, $\H_\infty(g) \leq 3kN^\frac{1}{k}$ readily follows.
\end{proof}

We now turn to the proof of Lemma~\ref{lem:cost_is_high}, and first introduce some general terminology. If $I = \{a,\ldots, b\}$ is an integer interval, we define its \textit{interior} as $\text{int}(I) := \{x \in V_N: a<x<b\}$. We let $E^\circ(I)$ be the set of edges of $E^\circ_N$ with both extremities belonging to $I$. For $0 \le u  < v \le 1$, we define $$\llbracket u, v\rrbracket := \{x \in V_N: uN \leq x \leq vN \};$$if $I$ is an integer interval, we define $$\llbracket u,v\rrbracket_I := \{x \in I: \min I + u|I| \leq x \leq \min I + v|I| \}.$$

From now on, we assume that
\begin{equation}\label{eq:set_of_assumptions}
\begin{array}{c} p\in[1,\infty],\;\;k \in \N,\;\;\eta \in \left(\frac{1}{k+1},\frac{1}{k}\right),\;\;\delta \in (0,1),\\[.2cm]g = (V_N,E_N) \in \msc G_N,\;\; \mathcal{H}_p(g) \leq N^\eta.\end{array}\end{equation}

Due to the assumption $\mathcal{H}_p(g) \leq N^\eta$, if we take $\sigma > \eta$, then we expect most pairs $x,y \in V_N$ to satisfy $\d_g(x,y) \leq N^\sigma$ (in case $p = \infty$, this in fact holds for $\sigma = \eta$ and \textit{all} pairs $x,y$). With this in mind, we fix $\sigma \geq \eta$ and introduce some additional terminology. We say that a vertex $x \in V_N$ is \textit{regular} if there exists $y \in V_N$ such that $|y-x| \geq N/4$ and $\d_g(x,y) \leq N^\sigma$. Vertex $x$ is \textit{irregular} if this does not hold, that is, if $\d_g(x,y) > N^\sigma$ for all $y$ with $|y-x| \geq N/4$. Note that for $p = \infty$ all vertices are regular. For $p < \infty$, we have
$$N^{p \eta} \geq \mathcal{H}(g)^p \geq \frac{1}{N^2}\sum_{\substack{x: x \text{ is }\\\text{irregular}}} \;\;\sum_{\substack{y: |y-x| \geq N/4}} \d_g^p(x,y) \geq \frac{1}{N^2} \cdot \frac{N}{2} \cdot N^{\sigma p} \cdot |\{x: x \text{ is irregular}\}|,$$
so that
\begin{equation}\label{eq:num_regular}
|\{x \in V_N: x \text{ is irregular}\}| \leq  2N^{1-p(\sigma-\eta)}.
\end{equation}

In the remainder of this section, the exponents $\eta$, $\delta$ and $\sigma$ will be held fixed, but $N$ will often be assumed to be large enough, possibly depending on $\eta$, $\delta$ and $\sigma$. 


Given $\Gamma \subset \msc{E}_N $ and $e = \{a,b\} \in E^\circ_N$, we define 
\begin{equation*}
\psi(e,\Gamma) :=  \mbox{ number of edges $e' = \{a',b'\}\in\Gamma\backslash E^\circ_N$ with $a' \leq a$ and $b' \geq b$}. 
\end{equation*}
In case $\psi(e, \Gamma) = n$, we say that the ground edge $e$ is covered $n$ times by $\Gamma$. Since $\gamma = 1$, for any $g = (V_N, E_N) \in \msc G_N$ we have \begin{equation}
\C(g) =\sum_{e \in E_N\backslash E^\circ_N} |e| = \sum_{e \in E^\circ_N} \psi(e,E_N). \label{eq:psi_covers}
\end{equation}

The proof of Lemma \ref{lem:cost_is_high} is split into three parts, called ``levels'', in which we progressively argue that ground edges are covered by long edges of $E_N$ (a ``long edge'' here is an edge $\{x,y\}$ with $|x-y| \geq N^\delta$). Level 1 (carried out in Lemma \ref{lem:new_aux2}) is a simple initializing estimate. Level 2 (in Lemma \ref{lem:new_aux_cost_high}) is obtained from recursively using Level 1, and identifies one layer in the pile of layers alluded to in the introduction. Level 3, which contains the statement of Lemma \ref{lem:cost_is_high}, is obtained from recursively using Level 2 to identify the correct number of layers present in the graph.

It will be helpful to describe heuristically the ideas of proof for the first two levels. Both Level 1 and 2 take as input an integer interval $I \subset V_N$ and state two alternatives, at least one of which must hold true for $I$. One of the alternatives is of the form ``$I$ has many irregular vertices'' and the other states that the ground edges of $I$ are covered by long edges of $E_N$ in a way which we deem satisfactory for that level. We will simultaneously treat the cases $p \in [1,\infty)$ and $p = \infty$, and the reader will note that the latter case is simpler, as irregular vertices are then absent and only one of the aforementioned alternatives is possible (namely, ground edges being satisfactorily covered).

For Level 1, the first alternative is that the middle third of $I$ only has irregular vertices. If this is not the case, then we can find a path of length less than $N^\sigma$ from the middle third of $I$ to the exterior of $I$. We then decompose $I = I'\cup I''$ in two subintervals, according to whether the path leaves $I$ from the left or the right (see Figure \ref{fig:decompi}). The idea is that we can guarantee that most ground edges of $I'$ are covered by long edges of the path, while we do not guarantee anything concerning $I''$.

\begin{figure}[htb]
\begin{center}
\setlength\fboxsep{0pt}
\setlength\fboxrule{0.5pt}
\fbox{\includegraphics[width = 1.0\textwidth]{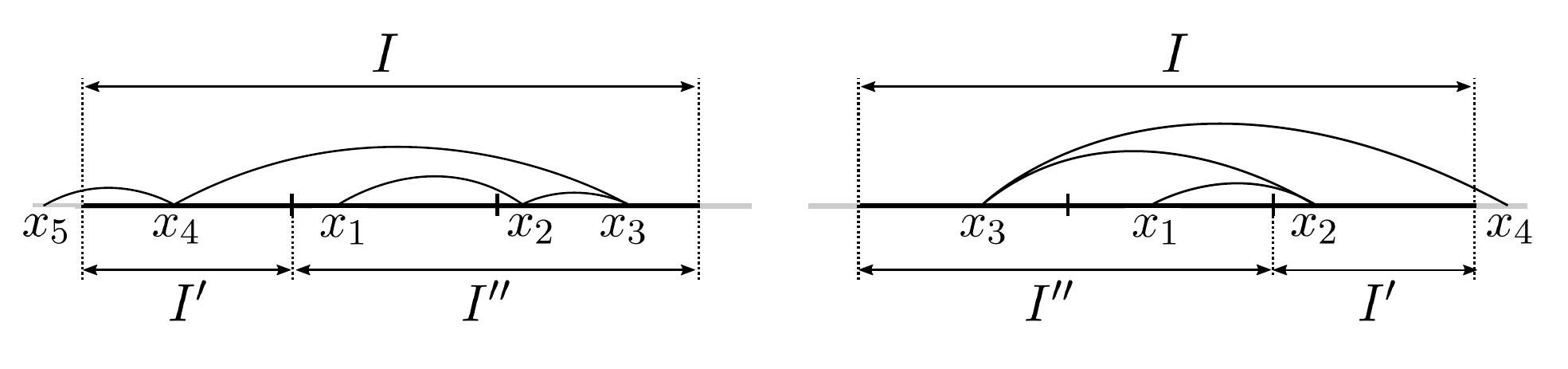}}
\end{center}
\caption{In case the depicted path leaves $I$ from the left, we let $I' = \llbracket 0,\frac13\rrbracket_I$ and $I'' = \llbracket \frac13, 1\rrbracket_I$. In case it leaves $I$ from the right, we let $I' = \llbracket \frac23, 1\rrbracket_I$ and $I''  =  \llbracket 0,\frac23\rrbracket_I$.}
\label{fig:decompi}
\end{figure}

\begin{lemma}\label{lem:new_aux2} \textbf{(Level 1 of recursion).}
If \eqref{eq:set_of_assumptions} holds and $N$ is large enough, then the following holds. For every  interval $I \subset V_N$ with $ |I| \leq N/4$, either
\begin{equation} \label{eq:bad_case_irregular}
|\{x\in I: x \text{ is irregular}\}| \geq |I|/4
\end{equation}
or there exist $\Phi_I \subset E_N$ and a decomposition $I = I' \cup I''$ of $I$ into intervals with disjoint interiors such that
\begin{align}
\label{eq:number_in_Phi}&|\Phi_I| \leq N^\sigma;\\
\label{eq:incidence_in_Phi}&\text{every edge of $\Phi_I$ is incident to at least one vertex of int($I$)};\\
\label{eq:length_in_Phi}&|e| \geq N^\delta \text{ for all } e \in \Phi_I;\\
\label{eq:size_of_I_prime}&|I''| \leq 3\cdot |I|/4;\\
\label{eq:coverage_in_Phi}&|\{e \in E^\circ(I'): \psi(e, \Phi_I) = 0\}| \leq 2N^{\sigma+\delta}.
\end{align}
\end{lemma}
\begin{proof}
We first note that, if $I$ is small (say, $|I| \leq N^{\sigma+\delta}$), then we can set $I' = I$ and $\Phi_I = I''= \varnothing$; then,   \eqref{eq:number_in_Phi}, \eqref{eq:incidence_in_Phi}, \eqref{eq:length_in_Phi},  \eqref{eq:size_of_I_prime} and \eqref{eq:coverage_in_Phi} are trivially satisfied. So let us assume that $|I| > N^{\sigma+\delta}$. We also assume that $N$ is large enough that $|\llbracket \frac13, \frac23 \rrbracket_I| \geq \frac{|I|}{4} $ and $|\llbracket \frac13,1\rrbracket_I| \leq \frac{3|I|}{4}$ for any $I$ with $|I| > N^{\sigma+\delta}$.

Suppose that \eqref{eq:bad_case_irregular} does not hold. Then, there exists a regular vertex $x\in\llbracket \frac13,\frac23\rrbracket_I$; since $|I|\leq N/4$, by the definition of regular there exists a path $\pi$ from $x$ to $(\text{int}(I))^c$ of length at most $N^\sigma$. We let $\Phi_I$ denote the set of edges $e$ in this path such that $|e| \geq N^\delta$. \eqref{eq:number_in_Phi} and \eqref{eq:incidence_in_Phi} then trivially hold. In case $\pi$ leaves $\text{int}(I)$ from the left, we let $I' = \llbracket 0,\frac13\rrbracket_I$ and $I'' = \llbracket \frac13,1\rrbracket_I$; in case $\pi$ leaves $\text{int}(I)$ from the right, we let $I'= \llbracket \frac23,1\rrbracket_I$ and $I'' = \llbracket 0,\frac23\rrbracket_I$. Then, \eqref{eq:size_of_I_prime} holds.

Let $e = \{x,y\} \in E^\circ(I')$ be a ground edge with $\psi(e, \Phi_I) = 0$. Then, there are two possibilities:
\begin{itemize}
\item $\pi$ traverses $e$. There can be no more than $N^\sigma$ edges for which this holds.
\item $\pi$ does not traverse $e$ and there is some edge $e' = \{x', y'\}$ with $1 < |e'| < N^\delta$ so that $\pi$ traverses $e'$ and $\psi(e,\{e'\}) = 1$, that is, $e'$ covers $e$. The number of edges  for which this is true is no more than
$$\sum_{\substack{e'': |e''| < N^\delta,\\\pi \text{ traverses } e''}} |e''| \leq N^{\sigma + \delta}.$$
\end{itemize} 
This proves \eqref{eq:coverage_in_Phi}.
\end{proof}

For Level 2, we again start with an interval $I$, which we re-label as $I_0$. We then apply the following procedure. We try to decompose $I_0 = I_0' \cup I_0''$ as in Level 1; if this is impossible (due to irregular vertices), we stop. Otherwise, we let $I_1 = I_0''$ and try to decompose $I_1 = I_1' \cup I_1''$ again as in Level 1; if this is impossible, we stop, etc., continuing until we are either forced to stop because  too many irregular vertices make a decomposition impossible, or we reach a sufficiently small interval $I_n$. 

In the statement below, the sets $\Gamma_I \subset E_N$ and $\tilde E_I \subset E^\circ(I)$ are the end products of this recursive procedure. $\Gamma_I$ is the set of all long edges obtained in successful decompositions (that is, a union of sets of the form $\Phi_{I_i}$ given by Lemma \ref{lem:new_aux2}). $\tilde E_I$ is the set of ground edges of $I$ which end up \textit{not} being covered by long edges of $\Gamma_I$. The alternatives in \eqref{eq:either_or} thus express that either $|\tilde E_I|$ is small or there are at least $\frac15 |\tilde E_I|$ irregular vertices in $I$.

\begin{lemma}\label{lem:new_aux_cost_high}\textbf{(Level 2 of recursion).}
If \eqref{eq:set_of_assumptions} holds and $N$ is large enough, then the following holds. For every  interval $I \subset V_N$ with $|I| < N/4$ there exist $\Gamma_I \subset E_N$ and $\tilde E_I \subset E^\circ(I)$ such that
\begin{align} \label{eq:number_in_Gamma}&|\Gamma_I| \leq N^\sigma (\log N)^2;\\
\label{eq:incidence_in_Gamma}&\text{every edge of }\Gamma_I\text{ is incident to at least one vertex of }\text{int}(I);\\
\label{eq:length_in_Gamma} &|e| \geq N^\delta \text{ for all } e \in \Gamma_I;\\
\label{eq:coverage_in_Gamma}&\psi(e, \Gamma_I) \geq 1\text{ for all }e \in E^\circ(I)\backslash \tilde E_I;\\
&\text{either }|\tilde E_I| \leq N^{\sigma+\delta}  (\log N)^4\text{ or }|\{x\in \text{int}(I): x \text{ is irregular}\}|\geq |\tilde E_I|/5.\label{eq:either_or}
\end{align}
\end{lemma}
\begin{proof} 
Denote $I_0 = I$. In case we have 
\begin{equation}\label{eq:bad_case_for0}|I_0| >  N^{\sigma+\delta} (\log N)^3 \quad \text{and}\quad |\{x \in \text{int}(I_0): x \text{ is irregular}\}| < |I_0|/4,\end{equation} we take $\Phi_{I_0} \subset E_N$, $I_0',I_0'' \subset I_0$ corresponding to $I_0$ as in Lemma \ref{lem:new_aux2} and denote $I_1 = I_0''$ (in particular, $|I_1| \leq 3|I_0|/4$). Next, if
$$|I_1| >  N^{\sigma+\delta} (\log N)^3 \quad \text{and}\quad |\{x \in \text{int}(I_1): x \text{ is irregular}\}| < |I_1|/4,$$
we take $\Phi_{I_1} \subset E_N$, $I_1',I_1'' \subset I_1$ corresponding to $I_1$ as in Lemma \ref{lem:new_aux2} and denote $I_2 = I_1''$ (in particular, $|I_2| \leq 3|I_1|/4$). We continue in this way, obtaining sets of vertices $I_0 \supset I_1 \supset \cdots$ and sets of edges $\Phi_{I_0},\;\Phi_{I_1},\ldots$ until we reach the first index $n$ for which either 
\begin{equation}
\label{eq:eq_case1} |I_n| \leq N^{\sigma+\delta} (\log N)^3
\end{equation}
or
\begin{equation}
\label{eq:eq_case2} |\{x \in \text{int}(I_n): x \text{ is irregular}\}| \geq |I_n|/4;
\end{equation}
note that, since $|I_{i+1}| \leq 3|I_i|/4$ for each $i$, we must have $n \leq (\log N)^2$ if $N$ is large enough.

Now, in case $n = 0$, then one of the conditions in \eqref{eq:bad_case_for0} fails; in either case, setting $\Gamma_I = \varnothing$ and $\tilde E_I = E^\circ(I)$, it can be readily seen that \eqref{eq:number_in_Gamma}, \eqref{eq:incidence_in_Gamma}, \eqref{eq:length_in_Gamma}, \eqref{eq:coverage_in_Gamma} and \eqref{eq:either_or} are all satisfied.

Assume $n > 0$. Let $$\Gamma_I = \cup_{i=0}^{n-1} \Phi_{I_i},\qquad \tilde E_I = E^\circ(I_n) \cup \left(\bigcup_{i=0}^{n-1}\{e \in E^\circ(I_i'): \psi(e, \Phi_{I_i}) = 0\}\right).$$ Then, \eqref{eq:number_in_Gamma} holds because $n \leq (\log N)^2$ (as already observed) and $|\Phi_{I_i}| \leq N^\sigma$ for each $i$ (as guaranteed in \eqref{eq:number_in_Phi}). Also, \eqref{eq:incidence_in_Gamma}, \eqref{eq:length_in_Gamma} and \eqref{eq:coverage_in_Gamma} respectively follow from \eqref{eq:incidence_in_Phi}, \eqref{eq:length_in_Phi} and the definition of $\tilde E_I$.

Let us prove \eqref{eq:either_or}. We observe that
\begin{equation}\label{eq:aux_to_end}\begin{split}
\tilde E_I \leq |I_n| + \sum_{i=0}^{n-1} |\{e \in E^\circ(I_i'): \psi(e, \Phi_{I_i}) = 0\}| &\stackrel{\eqref{eq:coverage_in_Phi}}{\leq} |I_n| + 2(n-1)N^{\sigma+\delta} \\&\hspace{.2cm}\leq |I_n| + 2N^{\sigma+\delta}(\log N)^2.
\end{split}\end{equation}
Assume that \eqref{eq:eq_case1} holds. Then, the above computation gives $\tilde E_I \leq N^{\sigma+\delta} (\log N)^3 + 2N^{\sigma+\delta}(\log N)^2 \leq N^{\sigma+\delta} (\log N)^4$, so we have \eqref{eq:either_or} in this case. 

Now assume that \eqref{eq:eq_case1} does not hold and \eqref{eq:eq_case2} holds. Since we then have $|I_n| > N^{\sigma+\delta} (\log N)^3 > 8N^{\sigma+\delta} (\log N)^2$,  \eqref{eq:aux_to_end} gives $|\tilde E_I| \leq \frac54 |I_n|$, and then
$$|\{x\in \text{int}(I): x \text{ is irregular}\} |\geq |\{x\in \text{int}(I_n): x \text{ is irregular}\} | \stackrel{\eqref{eq:eq_case2}}{\geq} \frac{1}{5}\cdot \frac{5|I_n|}{4} \geq \frac{|\tilde E_I|}{5}.$$
\end{proof}

\begin{lemma}\label{lem:new_aux_cost_high2}\textbf{(Level 3 of recursion).}
If \eqref{eq:set_of_assumptions} holds and $N$ is large enough, then there exist sets  $\Lambda_1 \subset \cdots \subset \Lambda_{k} \subset E_N$ such that, for every $j \in \{1,\ldots, k\}$, we have
\begin{eqnarray}
\label{eq:before_fin1}&&|\Lambda_j| \leq 3N^{\sigma j}(\log N)^{3j},\\
\label{eq:before_fin15}&&|e| \geq N^\delta \text{ for all } e \in \Lambda_j,\\
\label{eq:before_fin2}&&\hspace{-.18cm}\begin{array}{l}|\{e \in E^\circ_N: \psi(e, \Lambda_j) \geq j\}| \\[.1cm]\qquad \qquad\geq N - (\log N)^{5j}N^{\sigma j + \delta} -10j|\{x \in V_N: x\text{ is irregular}\}|.\end{array}
\end{eqnarray}
\end{lemma}
\begin{proof}
We will do induction on $j$. To start the induction, we fix $\bar x_1, \bar x_2, \bar x_3, \bar x_4 \in V_N$ so that, letting $K_0 = \llbracket 0, \bar x_1\rrbracket$, $K_1 = \llbracket \bar x_1, \bar x_2 \rrbracket$, $K_2 = \llbracket \bar x_2, \bar x_3\rrbracket$, $K_3 = \llbracket \bar x_3, \bar x_4\rrbracket$ and $K_4 = \llbracket \bar x_4, N-1\rrbracket$,  we have $|K_0|, |K_1|, |K_2|, |K_3|,|K_4| < N/4$. We then apply Lemma \ref{lem:new_aux_cost_high} to $K_0$, $K_1$, $K_2$, $K_3$ and $K_4$. For $i = 0,1,2,3,4$, let $\Gamma_{K_i} \subset E_N$ and $\tilde E_{K_i}$ be as in that lemma; then set $\Lambda_1 = \cup_{i=0}^5\Gamma_{K_i}$. Then, \eqref{eq:before_fin1} and \eqref{eq:before_fin15} with $j = 1$ respectively follow from \eqref{eq:number_in_Gamma} and \eqref{eq:length_in_Gamma}. By \eqref{eq:coverage_in_Gamma} we have
$$|\{e \in E^\circ_N:\psi(e,\Lambda_1) \geq 1\}| \geq N - \sum_{i=0}^4|\tilde E_{K_i}|;$$
furthermore, by \eqref{eq:either_or}, for each $i\in\{0,1,2,3,4\}$ we either have $|\tilde E_{K_i}| \leq N^{\sigma+\delta}(\log N)^4$ or $|\{x\in\text{int}(K_i):x\text{ is irregular}\}| \geq |\tilde E_{K_i}|/5$. Hence,
\begin{eqnarray}
\nonumber &&\sum_{i=0}^4 |\tilde E_{K_i}| \cdot \mathds{1}_{\{|\tilde E_{K_i}| \leq N^{\sigma + \delta}(\log N)^4\}} \leq 5 N^{\sigma + \delta}(\log N)^4, \\
&&\sum_{i=0}^4 |\tilde E_{K_i}| \cdot \mathds{1}_{\{|\tilde E_{K_i}| > N^{\sigma + \delta}(\log N)^4\}} \leq 5 |\{x \in V_N: x \text{ is irregular}\}|.\label{eq:bound:newdelta}
\end{eqnarray}
This proves \eqref{eq:before_fin2} with $j=1$.

Now assume $j < k$ and $\Lambda_j$ has been defined and satisfies  \eqref{eq:before_fin1}, \eqref{eq:before_fin15} and \eqref{eq:before_fin2}. Let $z_0 < z_1 < \cdots <z_{r}$ denote the vertices that belong to the set
$$\{0,\;\bar x_1,\; \bar x_2,\;\bar x_3,\;\bar x_4,\; N-1\} \cup \{x \in V_N: x \text{ is the extremity of some edge in } \Lambda_j\}.$$
Also define the integer intervals $I_i = \{z_{i-1},\ldots,z_{i}\}$, for $i \in \{1,\ldots, r\}$. Note that, for any fixed $i$, the value of $\psi(e, \Lambda_j)$ is the same for all edges $e$ contained in $E^\circ(I_i)$. We thus let $J_1,\ldots, J_s$ be those intervals among $I_1,\ldots, I_r$ that satisfy 
\begin{equation} \label{eq:what_Js_satisfy}\psi(e,\Lambda_j) \geq j \text{ for all }e \in J_i.\end{equation} We have 
\begin{eqnarray}\label{eq:val_r} &&s \leq r \leq 2|\Lambda_j| \stackrel{\eqref{eq:before_fin1}}{\leq} 6 N^{\sigma j}(\log N)^{3j},\\
&&\begin{split}\sum_{i=1}^s (|J_i|-1) &\stackrel{\eqref{eq:what_Js_satisfy}}{=} |\{e \in E^\circ_N: \psi(e, \Lambda_j) \geq j  \}|  \\&\stackrel{\eqref{eq:before_fin2}}{\geq} N - (\log N)^{5j}N^{\sigma j + \delta} -10j|\{x\in V_N:x \text{ is irregular}\}|.\end{split}\label{eq:good_size_appears}
\end{eqnarray}

Now, for each $i \in \{1,\ldots, s\}$ we apply Lemma \ref{lem:new_aux_cost_high} to $J_i$, thus obtaining $\Gamma_{J_i} \subset E_N$ and $\tilde E_{J_i} \subset E^\circ(J_i)$. Let $\Lambda_{j+1} = \Lambda_j \cup \left(\cup_{i=1}^s \Gamma_{J_i}\right)$. We observe that, for each $i$, 
\begin{equation}\label{eq:gamma_covered}\Gamma_{J_i} \cap \Lambda_j = \varnothing.\end{equation}Indeed, by construction no edge of $\Lambda_j$ is incident to vertices of $\text{int}(J_i)$, and by \eqref{eq:incidence_in_Gamma} every edge of $\Gamma_{J_i}$ is incident to at least one vertex of $\text{int}(J_i)$.

By \eqref{eq:length_in_Gamma}, all edges  $e \in \Lambda_{j+1}$ satisfy $|e| > N^\delta$; moreover, 
\begin{align*}|\Lambda_{j+1}| &= |\Lambda_j| + \sum_{i=1}^s |\Gamma_{J_i}| \\&\stackrel{\eqref{eq:number_in_Gamma},\eqref{eq:before_fin1},\eqref{eq:val_r}}{\leq} 3N^{\sigma j}(\log N)^{3j} +  6N^{\sigma j}(\log N)^{3j} \cdot N^\sigma (\log N)^2 \\&\qquad\leq N^{\sigma(j+1)} (\log N)^{3(j+1)}.\end{align*} 
Hence the proof will be complete once we show that
\begin{equation}
\label{eq:what_is_missing} \begin{split}&|\{e\in E^\circ_N: \psi(e, \Lambda_{j+1}) \geq j+1\}| \\&\geq  N - (\log N)^{5(j+1)}N^{\sigma (j+1) + \delta} -10(j+1)|\{x\in V_N: x \text{ is irregular}\}|.\end{split}
\end{equation}

Applying \eqref{eq:coverage_in_Gamma}, \eqref{eq:what_Js_satisfy} and \eqref{eq:gamma_covered} to each $\Gamma_{J_i}$,  we have
\begin{equation*}
i \in \{1,\ldots, s\},\;e \in E^\circ(J_i) \backslash \tilde E_{J_i} \Longrightarrow \psi(e, \Lambda_{j+1}) \geq \psi(e, \Lambda_j) + \psi(e, \Gamma_{J_i}) \geq j+1,
\end{equation*}
hence
$$\begin{aligned}|\{e\in E^\circ_N: \psi(e, \Lambda_{j+1}) \geq j+1\}| &\geq \sum_{i=1}^s (|J_i| - 1 - |\tilde E_{J_i}|) \\&\hspace{-3cm}\stackrel{\eqref{eq:good_size_appears}}{\geq} N - (\log N)^{5j}N^{\sigma j + \delta} -10j|\{x\in V_N: x \text{ is irregular}\}| - \sum_{i=1}^s |\tilde E_{J_i}|.\end{aligned}$$
We then conclude by bounding
\begin{align*}\sum_{i=1}^s |\tilde E_{J_i}| \cdot \mathds{1}_{\{|\tilde E_{J_i}| \leq N^{\sigma+\delta} (\log N)^4\}} &\stackrel{\eqref{eq:val_r}}{\leq} 6N^{\sigma j}(\log N)^{3j} \cdot N^{\sigma+\delta}(\log N)^4 \\&\leq N^{\sigma (j+1) + \delta} (\log N)^{3j+4} \end{align*}
and, similarly to \eqref{eq:bound:newdelta},
\begin{equation*}\sum_{i=1}^s |\tilde E_{J_i}| \cdot \mathds{1}_{\{|\tilde E_{J_i}| > N^{\sigma+\delta} (\log N)^4\}} \leq 10|\{x\in V_N: x \text{ is irregular}\}|. \qedhere
\end{equation*}
\end{proof}
\begin{proof}[Proof of Lemma \ref{lem:cost_is_high}]
Recall that so far our only assumption concerning $\sigma$ was that $\sigma \geq \eta$; now, setting $j = k$ in \eqref{eq:before_fin2}, we will choose the value of $\sigma$ that makes the estimate in \eqref{eq:before_fin2} the sharpest as $N \to \infty$.

In case $p = \infty$, we simply take $\sigma = \eta$ and the desired result follows, as there are no irregular vertices in this case. In case $p \in [1,\infty)$, using \eqref{eq:num_regular} we obtain
$$|\{e \in E^\circ_N: \psi(e, E_N) \geq k\}| \geq  N - (\log N)^{5k}N^{\sigma k + \delta} -10k N^{1-p(\sigma- \eta)}.$$
We then set $\sigma = \frac{1+p \eta - \delta}{k+p}$, so that we equate the two exponents of $N$:
$$\sigma k + \delta = 1-p(\sigma - \eta) = 1-\frac{p(1-k \eta - \delta)}{k+p}.$$
Note that $\sigma \geq \eta$ follows from our assumption that $\delta < 1-k \eta$. This completes the proof.
\end{proof}

\appendix
\section{Large deviation estimate}
\label{s.ldp-appendix}
The purpose of this appendix is to prove the following large deviation result concerning sums of i.i.d.\ random variables with stretched exponential tails. While the result is classical, and can be deduced for instance from the more general results of \cite{gantert}, we prefer to give a self-contained and short proof here for the reader's convenience.

\begin{proposition}
\label{p.ldp-stretch}
Let $\ga \in (0,1]$, $\theta > 0$ and $(X_i)_{i \in \N}$ be i.i.d.\ non-negative random variables satisfying $\E[\exp(\theta X_1^\ga)] < \infty$. For every $m > 0$, there exists $c > 0$ such that uniformly over $N \ge 1$,
$$
\P \Ll[ \sum_{i = 1}^N X_i \ge \Ll(\E[X_1] + m\Rr) N \Rr] \le \exp\Ll(-cN^{\ga}\Rr).
$$
\end{proposition}
\begin{proof}
By a change of variables, it suffices to prove the result with $\theta = 1$.
Let 
$$
\td X^{(N)}_i := X_i \wedge N, \ \quad \ov X^{(N)}_i := \td X^{(N)}_i - \E\Ll[\td X^{(N)}_i\Rr], 
$$
and
$$
\varphi_N(\lambda) := \E \Ll[\exp\Ll(\lambda \ov X^{(N)}_1\Rr) \Rr].
$$
We first show that there exists $C > 0$ such that uniformly over $\lambda \le N^{\ga - 1}/2$ and $N$ sufficiently large,
\begin{equation}
\label{e:control-exp}
\varphi_N(\lambda) \le \exp \Ll( C \lambda^2\Rr) .
\end{equation}
Since $\varphi_N'(0) = 0$, we have
$$
\varphi_N(\lambda) \le 1 + \lambda^2 \sup_{[0,\lambda]} \varphi_N''.
$$
In order to prove \eqref{e:control-exp}, it thus suffices to show that there exists $C < \infty$ such that uniformly over $\lambda \le N^{\ga - 1}/2$ and $N$ sufficiently large,
\begin{equation}
\label{e:control2}
\varphi_N''(\lambda) \le C .
\end{equation}
Clearly,
\begin{equation}
\label{e:control-mean}
\E\Ll[\td X^{(N)}_1\Rr] \xrightarrow[N \to \infty]{} \E[X_1] < \infty,
\end{equation}
so in particular, $\sup_N\Ll|\E\Ll[\td X^{(N)}_1\Rr] \Rr| < \infty$. Hence, 
\begin{align*}
\varphi_N''(\lambda) & \le \E \Ll[\Ll(\ov X^{(N)}_1\Rr)^2\exp\Ll(\lambda \ov X^{(N)}_1\Rr) \Rr] \\
& \le \E \Ll[\Ll(\td X^{(N)}_1 + C \Rr)^2\exp\Ll(\lambda \Ll(\td X^{(N)}_1 + C\Rr) \Rr) \Rr] \\
& \le \E \Ll[\Ll(X_1 + C \Rr)^2\exp\Ll(\lambda \Ll(N^{1-\gamma} X_1^\gamma + C\Rr) \Rr) \Rr],
\end{align*}
where in the last line, we used the fact that for every $x \ge 0$, $N^{1-\gamma} x^\gamma \ge x \wedge N$. 
We then obtain \eqref{e:control2}, and thus \eqref{e:control-exp} (uniformly over $\lambda \le N^{\gamma-1}/2$), using the fact that $\E[\exp(X_1^\ga)] < \infty$ and the elementary observation $\sup_{x \ge 0} x^2 \exp(x/2) / \exp(x) < \infty$.

In view of \eqref{e:control-mean}, we can choose $N$ sufficiently large that 
$$
\Ll|\E\Ll[\td X^{(N)}_1\Rr]  - \E[X_1]\Rr| \le \frac{m}{2}.
$$
For such a choice of $N$, we have
\begin{align*}
\P \Ll[ \sum_{i = 1}^N X_i \ge \Ll(\E[X_1] + m\Rr)N \Rr] & \le \P[\exists N : X_i \ge N] + \P \Ll[ \sum_{i = 1}^N \td X^{(N)}_i \ge \Ll(\E[X_1] + m\Rr) N \Rr] \\
& \le \P[\exists N : X_i \ge N] + \P \Ll[ \sum_{i = 1}^N \ov X^{(N)}_i \ge \frac{m}{2} N \Rr] .
\end{align*}
The integrability assumption and Chebyshev's inequality ensure that
$$
\P[\exists N : X_i \ge N] \le \exp\Ll(-N^{\ga}/2\Rr). 
$$
For the other term, Chebyshev's inequality and \eqref{e:control-exp} yield that for every $\lambda \le N^{\ga - 1}/2$,
$$
\P \Ll[ \sum_{i = 1}^N \ov X^{(N)}_i \ge \frac{m}{2} N \Rr] \le \exp \Ll[ \Ll( C\lambda^2 - \frac{\lambda m}{2}\Rr) N \Rr] .
$$
Choosing $\lambda = N^{\ga-1}/2$ leads to the announced result when $\ga < 1$; otherwise, it suffices to choose $\lambda > 0$ sufficiently small and independent of $N$.
\end{proof}

\bigskip

\noindent \textbf{Acknowledgments.} We would like to thank Florent Cadoux (G2ELab, ERDF Chair, Grenoble) for inspiring discussions on networks of electricity distribution, and the two referees and the associate editor for their very useful comments on an earlier version of this paper. We would also like to thank Emmanuel Jacob, Julia Komj\'athy, and Remco van der Hofstad for helpful discussions.

\bibliographystyle{abbrv}
\bibliography{Gibbs_graphs}

\end{document}